\pdfoutput=1
\RequirePackage{ifpdf}
\ifpdf 
\documentclass[pdftex]{sigma}
\else
\documentclass{sigma}
\fi

\numberwithin{equation}{section}

\newtheorem{Theorem}{Theorem}[section]
\newtheorem{Proposition}[Theorem]{Proposition}
 { \theoremstyle{definition}
\newtheorem{Example}[Theorem]{Example}}

\begin{document}

\allowdisplaybreaks

\newcommand{\arXivNumber}{1604.07847}

\renewcommand{\thefootnote}{}

\renewcommand{\PaperNumber}{025}

\FirstPageHeading

\ShortArticleName{Multi-Poisson Approach to the Painlev\'{e} Equations}

\ArticleName{Multi-Poisson Approach to the Painlev\'{e} Equations:\\
from the Isospectral Deformation\\ to the Isomonodromic Deformation\footnote{This paper is a~contribution to the Special Issue on Symmetries and Integrability of Dif\/ference Equations. The full collection is available at \href{http://www.emis.de/journals/SIGMA/SIDE12.html}{http://www.emis.de/journals/SIGMA/SIDE12.html}}}

\Author{Hayato CHIBA}

\AuthorNameForHeading{H.~Chiba}

\Address{Institute of Mathematics for Industry, Kyushu University, Fukuoka, 819-0395, Japan}
\Email{\href{mailto:chiba@imi.kyushu-u.ac.jp}{chiba@imi.kyushu-u.ac.jp}}

\ArticleDates{Received October 11, 2016, in f\/inal form April 11, 2017; Published online April 15, 2017}

\Abstract{A multi-Poisson structure on a Lie algebra $\mathfrak{g}$ provides a systematic way to construct completely integrable Hamiltonian systems on $\mathfrak{g}$ expressed in Lax form $\partial X_\lambda /\partial t = [X_\lambda , A_\lambda ]$ in the sense of the isospectral deformation, where $X_\lambda , A_\lambda \in \mathfrak{g}$ depend rationally on the indeterminate $\lambda $ called the spectral parameter. In this paper, a method for modifying the isospectral deformation equation to the Lax equation $\partial X_\lambda /\partial t = [X_\lambda , A_\lambda ] + \partial A_\lambda /\partial \lambda $ in the sense of the isomonodromic deformation, which exhibits the Painlev\'{e} property, is proposed. This method gives a few new Painlev\'{e} systems of dimension four.}


\Classification{34M35; 34M45; 34M55}

\renewcommand{\thefootnote}{\arabic{footnote}}
\setcounter{footnote}{0}

\section{Introduction}\label{section1}

A dif\/ferential equation def\/ined on a complex region is said to have the Painlev\'{e} property if any movable singularity of any solution is a pole. Painlev\'{e} and his group classif\/ied second order ODEs having the Painlev\'{e} property and found new six dif\/ferential equations called the Painlev\'{e} equations. Nowadays, it is known that they are written in Hamiltonian forms
\begin{gather*}
(\text{P}_\text{J})\colon \frac{dq}{dt} = \frac{\partial H_J}{\partial p}, \qquad
\frac{dp}{dt} = -\frac{\partial H_J}{\partial q}, \qquad J = \text{I}, \dots, \text{VI}.
\end{gather*}
Among six Painlev\'{e} equations, the Hamiltonian functions of the f\/irst, second and fourth Painlev\'{e} equations are polynomials in both of the independent variable~$t$ and the dependent variab\-les~$(q,p)$. They are given by
\begin{gather}
H_{\text{I}} = \frac{1}{2}p^2 - 2q^3 - tq, \label{2dimP1} \\
H_\text{II} = \frac{1}{2}p^2 - \frac{1}{2}q^4 - \frac{1}{2}tq^2 - \alpha q, \label{2dimP2} \\
H_\text{IV} = -pq^2 + p^2q - 2pqt - \alpha p + \beta q, \label{2dimP4}
\end{gather}
respectively, where $\alpha, \beta \in \mathbb{C}$ are arbitrary parameters. Another important property of the Painlev\'{e} equations is that they are expressed as Lax equations. Let $L_\lambda $ and $A_\lambda $ be square matrices which depend rationally on the indeterminate $\lambda $ called the spectral parameter. The Painlev\'{e} equations are written in Lax form as
\begin{gather}
\frac{\partial L_\lambda }{\partial t} = [L_\lambda , A_\lambda ] + \frac{\partial A_\lambda }{\partial \lambda },\label{Lax1}
\end{gather}
for some choice of $L_\lambda $ and $A_\lambda $. This equation arises from the compatibility condition of the two dif\/ferential systems
\begin{gather*}
\frac{\partial \Psi}{\partial \lambda } = L_\lambda \Psi, \quad \frac{\partial \Psi}{\partial t } = A_\lambda \Psi.
\end{gather*}
Since the monodromy of the former system $\partial \Psi/\partial \lambda = L_\lambda \Psi$ is independent of $t$ if the equation~(\ref{Lax1}) is satisf\/ied, (\ref{Lax1})~is called the isomonodromic deformation equation.

Another type of the Lax equation is of the form
\begin{gather}
\frac{\partial X_\lambda }{\partial t} = [X_\lambda , A_\lambda ],\label{Lax2}
\end{gather}
which is called the isospectral deformation equation because the eigenvalues of the matrix $X_\lambda $ is independent of~$t$. There are several systematic ways to construct isospectral deformation equations~\cite{AMV}. In particular, a Lie algebraic method have been often employed. Let $\mathfrak{g}$ be a Lie algebra. On the dual space $\mathfrak{g}^*$, there exists a canonical Poisson structure called the Lie--Poisson structure. If $\mathfrak{g}$ is equipped with a nondegenerate bilinear symmetric form, the Lie--Poisson structure is also def\/ined on~$\mathfrak{g}$. Let $P \colon T^*\mathfrak{g} \to T\mathfrak{g}$ be the Poisson tensor and $F \colon \mathfrak{g} \to \mathbb{C}$ a smooth function. Then, the vector f\/ield $PdF$ on $\mathfrak{g}$ can be expressed as the Lax equation~(\ref{Lax2}) with some $X_\lambda , A_\lambda \in \mathfrak{g}$~\cite{AMV}.

It is notable that the isospectral deformation equation (\ref{Lax2}) is completely integrable for most examples, although the isomonodromic deformation equation (\ref{Lax1}) is not in general; it is believed that solutions of an isomonodromic deformation equation def\/ine new functions called the Painlev\'{e} transcendents.

In Nakamura~\cite{Nak}, a way to obtain the isospectral deformation equation (\ref{Lax2}) from the isomonodromic deformation equation~(\ref{Lax1}) by a certain scaling of the time~$t$ is proposed, which is called the autonomous limit. She proved that the autonomous limits of 6-types of two dimensional Painlev\'{e} equations and 40-types of four dimensional Painlev\'{e} equations are completely integrable. Such relations of Painlev\'{e} systems with autonomous integrable systems are known between Gaudin model and Schlesinger system, and also found in~\cite{LO} (Painlev\'{e}--Calogero correspondence).

The purpose in the present paper is opposite; a way to construct the isomonodromic deformation equation (\ref{Lax1}) from the isospectral deformation equation (\ref{Lax2}) will be proposed. Let $\mathfrak{g}$ be a simple Lie algebra over~$\mathbb{C}$. Consider the set of $\mathfrak{g}$-valued polynomials of degree $n$
\begin{gather*}
\mathfrak{g}_n:= \big\{ X_\lambda := X_0 \lambda ^n + X_1\lambda ^{n-1} + \dots + X_n \, | \, X_i\in \mathfrak{g}\big\},
\end{gather*}
with the indeterminate $\lambda $. This set $\mathfrak{g}_n$ is equipped with a structure of a Lie algebra by a certain Lie bracket. At f\/irst, the isospectral deformation equation (\ref{Lax2}) on $\mathfrak{g}_n$ is constructed with the aid of the bi-Poisson theory of Magri et al.~\cite{Mag1, Mag2, Mag3, Mag4}. Isospectral deformation equations obtained in this method are shown to be completely integrable (Theorem~\ref{theorem2.4}). Next, we restrict the equations onto a symplectic leaf. Let $\varphi _1 , \dots ,\varphi _N$ be Casimir functions of an underlying Poisson structure on $\mathfrak{g}_n$.
A symplectic leaf $S$ is def\/ined by the level surface of them as
\begin{gather*}
S:=\{ \varphi _i= \alpha _i \, (\text{const}) \, | \, i=1, \dots ,N\}.
\end{gather*}
Restricted on the leaf $S$, the isospectral deformation equation (\ref{Lax2}) becomes an integrable Hamiltonian system. Since the matrix $X_\lambda \in \mathfrak{g}_n$ depends on the parameters $\alpha := (\alpha_1, \dots ,\alpha _N)$, it is denoted as $X_\lambda = X_\lambda (t, \alpha )$.

Now suppose that there exists a parameter, say $\alpha _{j}$, such that the following condition holds
\begin{gather}
\frac{\partial X_\lambda }{\partial \alpha _{j}}(t, \alpha ) = \frac{\partial A_\lambda }{\partial \lambda }.\label{1-8}
\end{gather}
Equation~(\ref{Lax2}) is put together with equation~(\ref{1-8}) to yield
\begin{gather*}
\frac{\partial X_\lambda}{\partial t}(t,\alpha ) + \frac{\partial X_\lambda }{\partial \alpha _{j}}(t, \alpha )
 = [X_\lambda , A_\lambda ]+\frac{\partial A_\lambda }{\partial \lambda }.
\end{gather*}
Def\/ine the Lax matrix $L_\lambda $ by
\begin{gather*}
L_\lambda := X_\lambda (t, \alpha )|_{\alpha _{j} = t},
\end{gather*}
where the parameter $\alpha _{j}$ satisfying the condition (\ref{1-8}) is replaced by~$t$. Then, the above equation is rewritten as the isomonodromic deformation equation (\ref{Lax1}).

Remark that the isomonodromic deformation equation (\ref{Lax1}) is equivalent to the zero curvature condition of the connection $1$ form $L_\lambda d\lambda + A_\lambda dt$ on a vector bundle over the $(t,\lambda )$-space, while the condition (\ref{1-8}) is the exactness condition of the connection $1$ form $X_\lambda d\lambda + A_\lambda d\alpha _j$.

This method is demonstrated for the following three cases (I) $\mathfrak{g} = \mathfrak{sl}(2, \mathbb{C}),\,n = 2$,
(II)~$\mathfrak{g} = \mathfrak{sl}(2, \mathbb{C}),\,n = 3$ and (III) $\mathfrak{g} = \mathfrak{so}(5, \mathbb{C}),\,n =1$.
For the case (I), the f\/irst, second and fourth Painlev\'{e} equations (\ref{2dimP1}), (\ref{2dimP2}), (\ref{2dimP4})
will be obtained in Section~\ref{section3}.

More generally, for $\mathfrak{g} = \mathfrak{sl}(2, \mathbb{C})$ with general $n$, it seems that one can obtain several Painlev\'{e} hierarchies of dimension $2n-2$, including the f\/irst Painlev\'{e} hierarchy $(\text{P}_\text{I})_m$~\cite{Koi, Kud, Shi}, the second-f\/irst Painlev\'{e} hierarchy~$(\text{P}_\text{II-1})_m$ \cite{CJM, CJP, Koi, Kud}, the second-second Painlev\'{e} hierarchy~$(\text{P}_\text{II-2})_m$ and the fourth Painlev\'{e} hierarchy~$(\text{P}_\text{IV})_m$ \cite{GJP, Koi}. They are $2m$-dimensional Hamiltonian PDEs of the form ($m=n-1$)
\begin{gather*}
\frac{\partial q_j}{\partial t_i} = \frac{\partial H_i}{\partial p_j},\qquad
\frac{\partial p_j}{\partial t_i} = -\frac{\partial H_i}{\partial q_j},\qquad j=1,\dots ,m, \qquad i=1,\dots ,m, \\
H_i = H_i(q_1, \dots ,q_m, p_1, \dots ,p_m, t_1, \dots ,t_m)
\end{gather*}
consisting of $m$ Hamiltonians $H_1, \dots ,H_m$ with $m$ independent variables $t_1, \dots ,t_m$. When $m=1$ (the case (I)), $(\text{P}_\text{I})_1$ and $(\text{P}_\text{IV})_1$ are reduced to the f\/irst and fourth Painlev\'{e} equations, respectively. Both of $(\text{P}_\text{II-1})_1$ and $(\text{P}_\text{II-2})_1$ coincide with the second Painlev\'{e} equation, while they are dif\/ferent systems for $m\geq 2$. When $m=2$ (the case~(II)), Hamiltonians of $(\text{P}_\text{I})_2$, $(\text{P}_\text{II-1})_2$, $(\text{P}_\text{II-2})_2$ and $(\text{P}_\text{IV})_2$ are given by
\begin{align}
& (\text{P}_\text{I})_2 && \begin{cases}
 H_1=2p_2p_1 + 3p_2^2q_1 + q_1^4 - q_1^2q_2 - q_2^2 - t_1q_1 + t_2\big(q_1^2 - q_2\big), \\
H_2= p_1^2 + 2p_2p_1q_1 - q_1^5 + p_2^2q_2 + 3q_1^3q_2 - 2q_1q_2^2 \\
\hphantom{H_2=}{} + t_1 \big(q_1^2-q_2\big) + t_2\big(t_2q_1 + q_1q_2 - p_2^2\big),
\end{cases}\label{4dimP1}\\
& (\text{P}_\text{II-1})_2 && \begin{cases}
 H_1 = 2p_1p_2 - p_2^3-p_1q_1^2 + q_2^2 - t_1p_2 + t_2p_1 + 2 \alpha q_1, \\
 H_2= -p_1^2 + p_1p_2^2 + p_1p_2q_1^2 + 2p_1q_1q_2 \\
\hphantom{H_2=}{} + t_1p_1 + t_2\big(t_2p_1 - p_1q_1^2 + p_1p_2\big) - \alpha (2p_2q_1 + 2q_2 + 2t_2q_1),
\end{cases}\label{4dimP21}\\
& (\text{P}_\text{II-2})_2 && \begin{cases}
 H_1=p_1p_2 - p_1q_1^2 - 2p_1q_2 + p_2q_1q_2 + q_1q_2^2 + q_2t_1 + t_2 (q_1q_2-p_1)+ \alpha q_1, \\
 H_2=p_1^2 - p_1p_2q_1 + p_2^2q_2 - 2p_1q_1q_2 - p_2q_2^2 + q_1^2q_2^2 \\
\hphantom{H_2=}{} + t_1( q_1q_2- p_1) - t_2\big(p_1q_1+q_2^2 + q_2t_2\big) + \alpha p_2,
\end{cases}\hspace*{-10mm}\label{4dimP22}\\
& (\text{P}_\text{IV})_2 && \begin{cases}
 H_1=p_1^2 + p_1p_2 - p_1q_1^2 + p_2q_1q_2- p_2q_2^2 - t_1p_1 + t_2p_2q_2 + \alpha q_2 + \beta q_1, \\
H_2= p_1p_2q_1 - 2p_1p_2q_2 - p_2^2 q_2 + p_2q_1q_2^2 \\
\hphantom{H_2=}{} + p_2q_2 t_1+ t_2\big(p_1p_2 - p_2q_2^2 + p_2q_2t_2\big) + (p_1-q_1q_2 + q_2t_2) \alpha - \beta p_2 ,
\end{cases}\hspace*{-10mm}\label{4dimP4}
\end{align}
respectively, with arbitrary parameters $\alpha, \beta \in \mathbb{C}$. These systems will be obtained from the case~(II) $\mathfrak{g} = \mathfrak{sl}(2, \mathbb{C})$, $n = 3$ in Section~\ref{section4}. In our method, such Hamiltonian PDEs are obtained if there are several Hamiltonian systems written in Lax form (\ref{Lax2}), and if there are several parameters satisfying~(\ref{1-8}); such parameters will be replaced by distinct times $t_1, t_2,\dots $.

We will f\/ind other $4$-dimensional Painlev\'{e} systems with Hamiltonian functions
\begin{gather}
H_{(1,1,2,0)} = -p_1^2q_1 - 2p_1q_1^2 + 2p_1q_2 - 2p_1p_2q_2 - 2p_2q_1q_2 \nonumber \\
\hphantom{H_{(1,1,2,0)} =}{} + (2p_1q_1 + 2p_2q_2)t + (2\alpha _2 + 2\beta_2 )q_1 + 2\beta_2 p_1 + 2\beta_3 p_2, \label{1120} \\
H_{(-1,4,1,2)} = p_1 - p_2^2 - 2p_1q_1q_2 - p_2q_2^2 + 2 \beta_3 q_2 + 2\beta_5 q_1 + p_2t, \label{-1412} \\
H_{\text{Cosgrove}} = -4p_1p_2 - 2p_2^2q_1 - \frac{73}{128}q_1^4 + \frac{11}{8}q_1^2q_2- \frac{1}{2}q_2^2 - q_1t - \frac{\alpha_2}{48}\left( q_1 + \frac{\alpha_2}{6} \right) q_1^2,\label{4dimCos}
\end{gather}
where $\alpha_i, \beta_i \in \mathbb{C}$ are arbitrary parameters (the subscripts for parameters are related to the weighted degrees so that the Hamiltonian functions become quasihomogeneous, see below). The f\/irst two systems will be also obtained from the case~(II). As far as the author knows, these systems have not appeared in the literature. The last one $H_{\text{Cosgrove}}$ will be obtained from the case~(III) $\mathfrak{g} = \mathfrak{so}(5, \mathbb{C})$, $n =1$ in Section~\ref{section5}. If we rewrite the system as a fourth order single equation of $q_1 = y$, we obtain
\begin{gather}
y'''' = 18 yy'' + 9(y')^2 - 24y^3 + 16t + \alpha y \left(y + \frac{1}{9}\alpha \right).\label{Cos}
\end{gather}
This equation was given in~\cite{Cos}, denoted by F-VI, without a proof that it has the Painlev\'{e} property. Since this system is obtained as the isomonodromic deformation equation in this paper, this equation actually enjoys the Painlev\'{e} property. In~\cite{Cos} it is conjectured that this equation def\/ines a new Painlev\'{e} transcendent (i.e., it is not reduced to known equations). Another expression of the Hamiltonian function of the same system is
\begin{gather}
\widetilde{H}_{\text{Cosgrove}} = 2p_1p_2 - \frac{18}{13}p_2^2q_1 - \frac{2}{169}q_1^4
 - \frac{180}{13}q_1^2q_2 + 6q_2^2 - 8q_1t + \frac{8}{9}\alpha _2q_1^3 + \frac{8}{27}\alpha _2^2q_1^2.\label{4dimCos2}
\end{gather}
The corresponding Hamiltonian system is also reduced to (\ref{Cos}).

Note that all of the Hamiltonian functions above are polynomials in both of the independent variables and the dependent variables. Furthermore, they are semi-quasihomogeneous functions. In general, a polynomial $H(x_1, \dots ,x_n)$ is called a quasihomogeneous polynomial if there are integers $a_1, \dots ,a_n$ and $h$ such that
\begin{gather}
H\big(\lambda ^{a_1}x_1 , \dots ,\lambda ^{a_n}x_n\big) = \lambda ^h H(x_1, \dots ,x_n)\label{quasi}
\end{gather}
for any $\lambda \in \mathbb{C}$. A polynomial $H$ is called a semi-quasihomogeneous if $H$ is decomposed into two polynomials as $H = H^P + H^N$, where $H^P$ satisf\/ies (\ref{quasi}) and $H^N$ satisf\/ies
\begin{gather*}
H^N\big(\lambda ^{a_1}x_1 , \dots ,\lambda ^{a_n}x_n\big) \sim o\big(\lambda ^h\big), \qquad |\lambda | \to \infty.
\end{gather*}
The integer $\operatorname{wdeg}(H) := h$ is called the weighted degree of $H$ with respect to the weight $\operatorname{wdeg}(x_1, \dots ,x_n) := (a_1 ,\dots ,a_n)$. For example, if we def\/ine degrees of variables by $\operatorname{wdeg} (q,p,t) = (2,3,4)$ for $H_{\text{I}}$, $\operatorname{wdeg} (q,p,t) = (1,2,2)$ for $H_{\text{II}}$ and $\operatorname{wdeg} (q,p,t) = (1,1,1)$ for $H_{\text{IV}}$, then Hamiltonian functions have the weighted degrees $6$, $4$ and $3$, respectively (Table~\ref{table1}). The weights for four dimensional systems above are shown in Table~\ref{table2}. In this paper, these weights are naturally obtained from a suitable def\/inition of weights of entries of a matrix $X_\lambda \in \mathfrak{g}_n$ and the spectral parameter~$\lambda $. In particular, the weights of the Hamiltonian functions are closely related to the exponents of simple Lie algebras because the Hamiltonian functions are essentially Ad-invariant polynomials of simple Lie algebras. See~\cite{Chi1, Chi2, Chi3} for the detailed study of the weights of the Painlev\'{e} equations.

\begin{table}[h]\centering\caption{Weights for two dimensional Painlev\'{e} equations.}\label{table1}\vspace{1mm}
\begin{tabular}{|c||c|c|}
\hline
 & $\operatorname{wdeg}(q,p,t)$ & $\operatorname{wdeg} (H)$ \\ \hline \hline
$\text{P}_\text{I}$ & $(2,3,4)$ & 6 \\ \hline
$\text{P}_\text{II}$ & $(1,2,2)$ & 4 \\ \hline
$\text{P}_\text{IV}$ & $(1,1,1)$ & 3 \\ \hline
\end{tabular}
\end{table}

\begin{table}[h]\centering\caption{Weights for four dimensional Painlev\'{e} equations.}\label{table2}\vspace{1mm}
\begin{tabular}{|c||c|c|c|}
\hline
 & $\operatorname{wdeg}(q_1,p_1,q_2,p_2)$ & $\operatorname{wdeg} (t_1, t_2)$ & $\operatorname{wdeg} (H_1, H_2)$ \\ \hline \hline
$(\text{P}_\text{I})_2$ & $(2,5,4,3)$ & $6,4$ & $8,10$ \\ \hline
$(\text{P}_\text{II-1})_2$ & $(1,4,3,2)$ & $4,2$ & $6,8$ \\ \hline
$(\text{P}_\text{II-2})_2$ & $(1,3,2,2)$ & $3,2$ & $5,6$ \\ \hline
$(\text{P}_\text{IV})_2$ & $(1,2,1,2)$ & $2,1$ & $4,5$ \\ \hline
$H_{(1,1,2,0)}$ & $(1,1,2,0)$ & $1$ & $3$ \\ \hline
$H_{(-1,4,1,2)}$ & $(-1,4,1,2)$ & $2$ & $4$ \\ \hline
$H_{\text{Cosgrove}}$ & $(2,5,4,3)$ & $6$ & $8$ \\ \hline
\end{tabular}
\end{table}

\section{Settings}\label{section2}

\subsection[Lie--Poisson structure on $\mathfrak{g}_n$]{Lie--Poisson structure on $\boldsymbol{\mathfrak{g}_n}$}\label{section2.1}

We def\/ine a multi-Poisson structure on a certain Lie algebra following Magri et al.~\cite{Mag1,Mag2,Mag3,Mag4}. Let $(\mathfrak{g}, [\, \cdot \,, \,\cdot \,])$ be a simple Lie algebra over $\mathbb{C}$. Consider the set of $\mathfrak{g}$-valued polynomials of degree~$n$
\begin{gather*}
\mathfrak{g}_n:= \big\{ X_\lambda := X_0 \lambda ^n + X_1\lambda ^{n-1} + \dots + X_n \, | \, X_i\in \mathfrak{g}\big\},
\end{gather*}
with the indeterminate $\lambda $. The bracket def\/ined by
\begin{gather*}
[X_\lambda ,Y_\lambda ]_n := [X_n, Y_n] + \lambda ([X_n, Y_{n-1}] + [X_{n-1}, Y_n]) + \cdots \\
\hphantom{[X_\lambda ,Y_\lambda ]_n :=}{} + \lambda ^n \left( [X_0, Y_n] + [X_1, Y_{n-1}] + \dots + [X_n, Y_0] \right)
\end{gather*}
introduces the structure of a Lie algebra on $\mathfrak{g}_n$. Note that $[X_\lambda , Y_\lambda ]_n$ coincides with $[X_\lambda , Y_\lambda ]$ expanded in $\lambda $ and truncated at degree~$n$.

It is known that the dual space $\mathfrak{g}^*$ of any Lie algebra $\mathfrak{g}$ is equipped with a~canonical Poisson structure called the Lie--Poisson structure. If a nondegenerate symmetric bilinear form \smash{$\eta \colon \mathfrak{g}\times \mathfrak{g} \to \mathbb{C}$} is def\/ined on $\mathfrak{g}$,
it induces a Lie--Poisson structure on $\mathfrak{g}$. For functions $F, G\colon \mathfrak{g} \to \mathbb{C}$, the Poisson bracket on $\mathfrak{g}$ is def\/ined by $\{ F, G\}(X) = \eta (X, [\nabla F(X), \nabla G(X)])$, where $\nabla F(X) \in \mathfrak{g}$ is def\/ined through $(dF)_X(Y) = \eta (\nabla F(X), Y)$.
To give the Lie--Poisson structure on $\mathfrak{g}_n$, we def\/ine a nondegenerate symmetric bilinear form $\eta$ on $\mathfrak{g}_n$ by
\begin{gather*}
\eta (X_\lambda , Y_\lambda ):= \sum^n_{i=0} \operatorname{Tr}(X_iY_{n-i}),
\end{gather*}
by which $\mathfrak{g}_n$ is identif\/ied with its dual. For a smooth function $F\colon \mathfrak{g}_n \to \mathbb{C}$, def\/ine the gradient $\nabla F \in \mathfrak{g}_n$ through $ (dF)(Y_\lambda ) = \eta (\nabla F, Y_\lambda )$, and def\/ine $\nabla _iF\in \mathfrak{g}$ by
\begin{gather*}
 \nabla F = (\nabla _n F) \lambda ^n + (\nabla _{n-1}F) \lambda ^{n-1} + \dots + \nabla _0 F.
\end{gather*}
Using them, the Lie--Poisson bracket on $\mathfrak{g}_n$ is given by
\begin{gather*}
\{ F, G\}_0 := \eta (X_\lambda , [\nabla F, \nabla G]_n) \\
\hphantom{\{ F, G\}_0}{} = \operatorname{Tr}(X_0 \cdot [\nabla _0F, \nabla _0G])
 + \operatorname{Tr}\left( X_1\cdot ( [\nabla _0F, \nabla _1G]+[\nabla _1F, \nabla _0G]) \right) +\cdots \\
\hphantom{\{ F, G\}_0=}{} + \operatorname{Tr} (X_n\cdot ( [\nabla _0F, \nabla _nG] + \dots + [\nabla _nF, \nabla _0G])) \\
\hphantom{\{ F, G\}_0}{}= -\operatorname{Tr} (\nabla _0F \cdot ([X_0, \nabla _0G]+[X_1, \nabla _1G]+ \dots + [X_n, \nabla _nG])) - \cdots \\
\hphantom{\{ F, G\}_0=}{} -\operatorname{Tr}(\nabla _{n-1}F \cdot ([X_{n-1}, \nabla _0G] + [X_n, \nabla _1G])) - \operatorname{Tr} (\nabla _nF \cdot [X_n, \nabla _0G]).
\end{gather*}
The Poisson tensor (bivector) $P_0 \colon T^*\mathfrak{g}_n \to T\mathfrak{g}_n$ is def\/ined so that
\begin{gather*}
\{ F, G\}_0 = dF (P_0dG) = \eta (\nabla F, P_0dG) = \sum^n_{i=0} \operatorname{Tr}(\nabla _iF \cdot (P_0dG)_i).
\end{gather*}
This implies
\begin{gather*}
-(P_0dG)_0 = [X_0, \nabla _0G]+[X_1, \nabla _1G]+ \dots + [X_n, \nabla _nG],\\
\cdots\cdots\cdots\cdots\cdots\cdots\cdots\cdots\cdots\cdots\cdots\cdots\cdots\cdots\cdots\\
-(P_0dG)_{n-1} = [X_{n-1}, \nabla _0G] + [X_n, \nabla _1G] ,\\
-(P_0dG)_{n} = [X_{n}, \nabla _0G].
\end{gather*}
The following expression is useful
\begin{gather*}
P_0\colon \ dG \mapsto - \left(
\begin{matrix}
[X_0, \,\cdot \,] & [X_1, \,\cdot \,] & \dots & [X_n, \,\cdot \,] \\
\vdots & & \raisebox{-0.1cm}{\rotatebox{75}{$\ddots$}} & \\
 [X_{n-1}, \,\cdot \,] & [X_n, \,\cdot \,] & & \\
 [X_{n}, \,\cdot \,]& & &
\end{matrix}
\right) \left(
\begin{matrix}
\nabla _0G \\
\vdots \\
\nabla _{n-1}G \\
\nabla _nG
\end{matrix}\right) \\
\hphantom{P_0\colon \ dG \mapsto}{}
 = \left(
\begin{matrix}
[\nabla _0G, X_0]+[\nabla _1G, X_1] + \dots + [\nabla _nG, X_n] \\
\vdots \\
 [\nabla _0G, X_{n-1}]+[\nabla _1G, X_n] \\
 [\nabla _0G, X_n]
\end{matrix}\right).
\end{gather*}
It is also represented as a matrix as follows. Let $A = A(X)$ be a representation matrix of the mapping
\begin{gather*}
T^* \mathfrak{g} (\simeq \mathfrak{g}) \to \mathfrak{g}, \qquad dG \mapsto [X, \nabla G],
\qquad G\colon \ \mathfrak{g} \to \mathbb{C},\qquad X\in \mathfrak{g}
\end{gather*}
with respect to some coordinates on $\mathfrak{g}$ (here $\nabla G$ is the gradient on $\mathfrak{g}$). By the def\/inition, $-A$~is a~Poisson tensor of the Lie--Poisson structure on $\mathfrak{g}$. Since $A(X)$ is linear in $X$, $A(X_\lambda )$ is expanded as $A(X_\lambda ) = \lambda ^n A(X_0) + \lambda ^{n-1}A(X_1) + \dots + A(X_n)$. Putting $A(X_j) = A_j$, $P_0$ is represented as an $(n+1)\dim (\mathfrak{g}) \times (n+1)\dim (\mathfrak{g})$ matrix
\begin{gather*}
P_0 = - \left(
\begin{matrix}
A_0 & \dots & A_{n-1} & A_n \\
A_1 & \dots & A_n & \\[0.0cm]
\vdots & \raisebox{-0.1cm}{\rotatebox{75}{$\ddots$}} & & \\
A_n & & &
\end{matrix}
\right).
\end{gather*}

In what follows, suppose $\dim (\mathfrak{g}) = d$, $\operatorname{rank} (\mathfrak{g}) = h$ and let $m_1 ,\dots ,m_h$ be exponents of~$\mathfrak{g}$.
Let $(y_1,\dots ,y_d)$ be coordinates on~$\mathfrak{g}$. It is known that the Casimir functions of the Lie--Poisson structure on~$\mathfrak{g}$ (i.e., a~function $\varphi $ satisfying $\{ F, \varphi \} = 0$ for any $F\colon \mathfrak{g} \to \mathbb{C}$) are the Ad-invariant polynomials denoted by $\varphi _i (y_1, \dots ,y_d)$, $i=1,\dots ,h$, and they satisfy $\deg (\varphi _i) = m_i + 1$.

Let $x_j := (x_{j,1}, \dots ,x_{j,d})$ be coordinates on the $j$-th copy of $\mathfrak{g}$ (coordinate expression for $X_j$) and $(x_0, \dots ,x_n)$ coordinates on $\mathfrak{g}_n$. We def\/ine the weighted degrees of variables to be
\begin{gather*}
\operatorname{wdeg}(x_j) = \operatorname{wdeg}(x_{j,\alpha }) = j, \qquad \operatorname{wdeg}(\lambda )=1.
\end{gather*}
Then, $X_\lambda $ is quasihomogeneous (homogeneous in the weighted sense) of $\operatorname{wdeg}(X_\lambda ) = n$. Substituting $y_\alpha = x_{0,\alpha }\lambda ^n + x_{1,\alpha }\lambda ^{n-1} + \dots +x_{n,\alpha }$ into $\varphi _i(y_1, \dots ,y_d)$ and expanding it in $\lambda $ provide
\begin{gather*}
\varphi _i(y_1, \dots ,y_d) = \varphi _{i,0}(x_0, \dots ,x_n) \lambda ^{(m_i+1)n}+\varphi _{i,1}(x_0, \dots ,x_n) \lambda ^{(m_i+1)n-1}+ \cdots\\
\hphantom{\varphi _i(y_1, \dots ,y_d) =}{} + \varphi _{i,(m_i+1)n}(x_0, \dots ,x_n), \qquad i=1,\dots ,h,
\end{gather*}
which def\/ines polynomials $\varphi _{i,j}$ on $\mathfrak{g}_n$ satisfying
\begin{gather*}
\deg (\varphi _{i,j}) = m_i+1,\qquad \operatorname{wdeg}(\varphi _{i,j}) = j.
\end{gather*}

\begin{Proposition}\label{proposition2.1}\quad
\begin{enumerate}\itemsep=0pt
\item[$(i)$] $\varphi _{i,j}$ depends only on $(x_0, \dots ,x_j)$ for $0\leq j\leq n-1$.
\item[$(ii)$] $\varphi _{i,j}(x_0, x_1,\dots ,x_n) = \varphi _{i, (m_i+1)n-j} (x_n, \dots ,x_1,x_0)$.
\item[$(iii)$] For each $i$, $j$, $\alpha $, the derivative $\partial \varphi _{i,j+k}/\partial x_{k,\alpha }$ is independent of $k=0,\dots ,n$.
\item[$(iv)$] For each $i$, $j$, the gradient $\nabla _k \varphi _{i,j+k}$ is independent of $k=0,\dots ,n$.
\item[$(v)$] For each $i$, $j$, $k$, the equality
\begin{gather*}
\sum^n_{l=0}A_l \frac{\partial \varphi _{i,j+k-l}}{\partial x_k}
 = \sum^n_{l=0}A_l \frac{\partial \varphi _{i,j-l}}{\partial x_0} = 0
\end{gather*}
holds.
\item[$(vi)$] The Casimir functions of the Lie--Poisson structure $P_0$ on $\mathfrak{g}_n$ are
\begin{gather*}
\varphi _{i,(m_i+1)n-j}, \qquad i=1,\dots ,h,\qquad j=0,\dots ,n.
\end{gather*}
\end{enumerate}
\end{Proposition}
\begin{proof}
(i) and (ii) follow from the def\/inition of $\varphi _{i,j}$.

(iii) For $y_\alpha = \sum\limits^n_{k=0}\lambda ^{n-k}x_{k,\alpha }$, we have
\begin{gather*}
\frac{\partial \varphi _i}{\partial y_\alpha }= \frac{\partial x_{k,\alpha }}{\partial y_\alpha }\frac{\partial }{\partial x_{k,\alpha }}
 \sum^{(m_i+1)n}_{j=0}\lambda ^{(m_i+1)n-j}\varphi _{i,j} \\
\hphantom{\frac{\partial \varphi _i}{\partial y_\alpha }}{}
= \sum^{(m_i+1)n}_{j=0} \lambda ^{m_in-j+k}\frac{\partial \varphi _{i,j}}{\partial x_{k,\alpha }}
= \sum^{m_in + k}_{j=k} \lambda ^{m_in-j+k}\frac{\partial \varphi _{i,j}}{\partial x_{k,\alpha }}.
\end{gather*}
For the last equality, we used part~(i) combined with part~(ii). Thus we obtain
\begin{gather*}
\frac{\partial \varphi _i}{\partial y_\alpha } = \sum^{m_in}_{j=0} \lambda ^{m_in-j}\frac{\partial \varphi _{i,j+k}}{\partial x_{k,\alpha }}.
\end{gather*}
Since the left hand side is independent of $k$, so is each coef\/f\/icient of $\lambda ^{m_in-j}$ in the right hand side. Part~(iv) immediately follows from~(iii).

(v) The f\/irst equality is a consequence of part~(iii). Since $\varphi _i(y)$ is a Casimir function of the Lie--Poisson structure on $\mathfrak{g}$,
$\operatorname{Ad}\varphi _i =0$, where $A$ is a matrix def\/ined before. Substituting $y = \sum\limits^n_{k=0}\lambda ^{n-k}x_k$ yields{\samepage
\begin{gather*}
\begin{split}
& 0= A \frac{\partial \varphi _i}{\partial y}
= \big(\lambda ^nA_0 + \lambda ^{n-1}A_1 + \dots + A_n\big) \sum^{m_in}_{j=0} \lambda ^{m_in-j}\frac{\partial \varphi _{i,j+k}}{\partial x_{k}} \\
& \hphantom{0}{} = \sum_{j,l} \lambda ^{m_in-j+n-l}A_l \frac{\partial \varphi _{i,j+k}}{\partial x_{k}}
= \sum^{m_in+l}_{j=l} \lambda ^{m_in+n-j} \sum^n_{l=0} A_l\frac{\partial \varphi _{i,j+k-l}}{\partial x_{k}}.
\end{split}
\end{gather*}
This proves the second equality of (v).}

To prove (vi), it is suf\/f\/icient to show
\begin{gather*}
\left(
\begin{matrix}
A_0 & \dots & A_{n-1} & A_n \\
A_1 & \dots & A_n & \\
\vdots & \raisebox{-0.1cm}{\rotatebox{75}{$\ddots$}} & & \\
A_n & & &
\end{matrix}
\right)
\left(
\begin{matrix}
\partial \varphi _{i,(m_i+1)n-j}/\partial x_0 \\
\partial \varphi _{i,(m_i+1)n-j}/\partial x_1 \\
\vdots \\
\partial \varphi _{i,(m_i+1)n-j}/\partial x_n
\end{matrix}
\right) = 0
\end{gather*}
for $j=0,\dots ,n$. This is verif\/ied with the aid of part~(v).
\end{proof}

\begin{Example}\label{example2.2}
For $\mathfrak{g} = \mathfrak{sl} (2, \mathbb{C})$, we have $d=3$, $h=1$ and $m_i = m_1 = 1$. Denote a general element $X_\lambda \in \mathfrak{g}_n$ as
\begin{gather*}
X_\lambda = \lambda ^nX_0 + \lambda ^{n-1} X_1 + \dots +X_n \\
\hphantom{X_\lambda}{} = \lambda ^n \left(
\begin{matrix}
u_0 & v_0 \\
w_0 & -u_0
\end{matrix}
\right) + \lambda ^{n-1}\left(
\begin{matrix}
u_1 & v_1 \\
w_1 & -u_1
\end{matrix}
\right) + \dots + \left(
\begin{matrix}
u_n & v_n \\
w_n & -u_n
\end{matrix}
\right).
\end{gather*}
Let $(u_j, v_j, w_j)$ be coordinates on the $j$-th copy of $\mathfrak{g}$ and $(u_0,v_0,w_0, \dots ,u_n,v_n,w_n)$ coordinates on~$\mathfrak{g}_n$. Then,
\begin{gather*}
\nabla _j F = \left(
\begin{matrix}
\displaystyle \frac{1}{2}\frac{\partial F}{\partial u_j} & \displaystyle \frac{\partial F}{\partial w_j} \vspace{1mm}\\
\displaystyle \frac{\partial F}{\partial v_j} & \displaystyle -\frac{1}{2}\frac{\partial F}{\partial u_j}
\end{matrix}
\right), \qquad A_j = \left(
\begin{matrix}
0 & v_j & -w_j \\
-v_j & 0 & 2u_j \\
w_j & -2u_j & 0
\end{matrix}
\right).
\end{gather*}
The Casimir function on $\mathfrak{g}$ is given by $\varphi _i = \varphi = u^2 + vw$. Then, the functions $\varphi _{i,j} = \varphi _j$ are def\/ined by expanding
\begin{gather*}
\big(\lambda ^n u_0 + \dots + u_n\big)^2 + \big(\lambda ^n v_0 + \dots + v_n\big)\big(\lambda ^n w_0 + \dots + w_n\big)
\end{gather*}
in $\lambda $. This gives
\begin{gather*}
\varphi _j = \sum_{k+l = j} \left( u_ku_l + v_k w_l \right), \qquad j=0,\dots ,2n.
\end{gather*}
Note that they are coef\/f\/icients of $-\det X_\lambda $. The Casimir functions of $\mathfrak{g}_n$ are given by $\varphi _j$ for $j=n,\dots ,2n$.
\end{Example}

\subsection[Multi-Poisson structure on $\mathfrak{g}^0_n$]{Multi-Poisson structure on $\boldsymbol{\mathfrak{g}^0_n}$}\label{section2.2}

In general, a manifold $M$ is called a bi-Poisson manifold if
\begin{enumerate}\itemsep=0pt
\item[(i)] there are two Poisson brackets $\{ \,\, , \,\, \}_0$ and $\{ \,\, , \,\, \}_1$, and
\item[(ii)] the linear combination $\{ \,\, , \,\, \}_0 + t \{ \,\,, \,\, \}_1$ is also a Poisson bracket for any $t\in \mathbb{C}$.
\end{enumerate}

See \cite{Mag1, Mag2, Mag3, Mag4} for applications of bi-Poisson manifolds to integrable systems. Here, we introduce a bi-Poisson structure on~$\mathfrak{g}_n$ following~\cite{Mag2}. The shift operator $X_\lambda \mapsto X_{\lambda +t}$ def\/ines an automorphism of $\mathfrak{g}_n$ with a~parameter $t\in \mathbb{C}$. It induces a deformation, denoted by $\{ \,\, , \,\,\}_t$, of the Lie--Poisson bracket $\{ \,\, , \,\, \}_0$. Let
\begin{gather*}
\{ \,\, , \,\, \}_t = \{ \,\, , \,\, \}_0 + t \{ \,\, , \,\, \}_1 + \dots + t^{n+1} \{ \,\, , \,\, \}_{n+1} + \cdots
\end{gather*}
be its expansion. Magnano and Magri \cite{Mag2} proved that each $\{ \,\,, \,\, \}_i$, $i = 0, \dots ,n+1$, and their any linear combination satisfy the axiom of a Poisson bracket. Hence, $\mathfrak{g}_n$ has $n+2$ compatible Poisson brackets and it becomes a multi-Poisson manifold. Their Poisson tensors are
\begin{gather*}
P_1 = \left(
\begin{array}{@{\,}c|cccc@{\,}}
0 & 0 &0 & \dots & 0 \\ \hline
0 &-A_1 & -A_2 & \dots & -A_n \\[0.1cm]
\vdots & \vdots & \vdots & \rotatebox{80}{$\ddots$} & \\
0 & -A_{n-1} & -A_n & & \\
0 &-A_n & & & \\
\end{array}
\right), \\
 P_{k+1} = \left(
\begin{array}{@{\,}c|ccc|ccc@{\,}}
0 & 0 & \dots & 0 & 0 & \dots & 0 \\ \hline
0 & & & A_0 & & & \\
\vdots & & \rotatebox{80}{$\ddots$} & \vdots & & & \\
0 & A_0 & \dots & A_{k-1} & & & \\ \hline
0& & & & -A_{k+1} & \dots & -A_n \\
\vdots & & & & \vdots & \rotatebox{80}{$\ddots$} & \\
0& & & & -A_n & & \\
\end{array}
\right), \qquad k=1,\dots ,n-1, \\
 P_{n+1} = \left(
\begin{array}{@{\,}c|cccc@{\,}}
0 & 0 & \dots & 0 & 0 \\ \hline
0 & & & & A_0 \\
0 & & & A_0 & A_1 \\[0.1cm]
\vdots & & \rotatebox{80}{$\ddots$} & \vdots & \vdots \\
0 & A_0 & \dots & A_{n-2} & A_{n-1} \\
\end{array}
\right).
\end{gather*}
($P_0$ is the same as before). Let $\mathfrak{g}^0_n$ be a submanifold of $\mathfrak{g}_n$ def\/ined by $x_0 = \text{const}$;
\begin{gather*}
\mathfrak{g}_n^0:= \big\{ X_\lambda = X_0 \lambda ^n + X_1\lambda ^{n-1} + \dots + X_n \, | \, X_0 = \text{const}\big\} \subset \mathfrak{g}_n.
\end{gather*}
Since the f\/irst row and column of $P_1, \dots ,P_{n+1}$ are zero (i.e., $x_0 = (x_{0,1}, \dots ,x_{0,d})$ are Casimir functions of them), the restrictions of them on $\mathfrak{g}^0_n$ def\/ine a multi-Poisson structure on $\mathfrak{g}^0_n$, whose brackets and tensors are again denoted by $(\{ \,\,, \,\, \}_i, P_i)$. The tensors are given by
\begin{gather*}
 P_1 = \left(
\begin{array}{@{}cccc@{}}
-A_1 & -A_2 & \dots & -A_n \\[0.1cm]
 \vdots & \vdots & \rotatebox{80}{$\ddots$} & \\
 -A_{n-1} & -A_n & & \\
-A_n & & & \\
\end{array}
\right), \\
 P_{k+1} = \left(
\begin{array}{@{}ccc|ccc@{}}
 & & A_0 & & & \\
 & \rotatebox{80}{$\ddots$} & \vdots & & & \\
 A_0 & \dots & A_{k-1} & & & \\ \hline
 & & & -A_{k+1} & \dots & -A_n \\
 & & & \vdots & \rotatebox{80}{$\ddots$} & \\
 & & & -A_n & & \\
\end{array}
\right), \qquad k=1,\dots ,n-1, \\
P_{n+1} = \left(
\begin{array}{@{}cccc@{}}
 & & & A_0 \\
 & & A_0 & A_1 \\[0.1cm]
 & \rotatebox{80}{$\ddots$} & \vdots & \vdots \\
 A_0 & \dots & A_{n-2} & A_{n-1} \\
\end{array}
\right).
\end{gather*}

For $i=1,\dots ,h$ and $j=1,\dots ,(m_i+1)n$, def\/ine functions $\psi_{i,j}$ on $\mathfrak{g}^0_n$ by{\samepage
\begin{gather*}
\psi_{i,j}(x_1, \dots ,x_n) := \varphi _{i,j}|_{\mathfrak{g}_n^0} = \varphi _{i,j}|_{x_0 = \text{const}}
\end{gather*}
(we do not def\/ine $\psi_{i,0}$ because $\varphi _{i,0}$ is constant on $\mathfrak{g}_n^0$).}

\begin{Proposition}\label{proposition2.3}\quad
\begin{enumerate}\itemsep=0pt
\item[$(i)$] Casimir functions of $P_{k+1}$ are $\psi_{i,j}$, $i=1,{\dots},h$, for $j=1,2,{\dots},k$ and for \smash{$j=m_in+k+1,$} $m_in+k+2,\dots ,(m_i+1)n$.

\item[$(ii)$] Casimir functions of the combination $\lambda P_{k+1} - P_k$ are $\psi_{i,j}$, $i=1,\dots ,h$, for $j=1,2,\dots ,k-1$ and for $j=m_in+k+1, m_in+k+2,\dots ,(m_i+1)n$, and
\begin{gather*}
\lambda ^{m_in}\psi_{i,k} + \lambda ^{m_in-1}\psi_{i, k+1}+ \dots +\psi_{i,m_in+k}, \qquad i=1,\dots ,h.
\end{gather*}
\item[$(iii)$] Let $F\colon \mathfrak{g}^0_n \to \mathbb{C}$ be a smooth function. The differential equation for the vector field $(\lambda P_{k+1} - P_k)dF$ is expressed in
Lax form as
\begin{gather*}
\frac{d}{dt}X_\lambda = [X_\lambda , \nabla _kF], \qquad X_\lambda =\lambda ^nX_0+\lambda ^{n-1}X_1+\dots +X_n.
\end{gather*}
\item[$(iv)$] Define the function $G_{i,k,j}$ to be
\begin{gather*}
G_{i,k,j} = -\big( \lambda ^{j-1}\psi_{i,k} + \lambda ^{j-2} \psi_{i,k+1} + \dots + \psi_{i,k+j-1}\big).
\end{gather*}
Then, the equality
\begin{gather*}
P_{k+1}d\psi_{i,k+j} = P_kd\psi_{i,k+j-1} = (\lambda P_{k+1}-P_k)dG_{i,k,j}
\end{gather*}
holds for $i=1,\dots ,h$, $j=1,\dots ,m_in$ and $k=1,\dots ,n$. In particular, the vector field $P_{k+1}d\psi_{i,k+j}$ is independent of~$k$ and the equation for it is expressed in Lax form as
\begin{gather*}
\frac{d}{dt}X_\lambda = [X_\lambda , \nabla _kG_{i,k,j}].
\end{gather*}
\item[$(v)$] The vector fields $P_{k+1}d\psi_{i,k+j}$ for $i=1,\dots ,h$ and $j=1,\dots , m_in$ commute with each other $($note that it is zero when $j\notin \{ 1,\dots ,m_in\})$.
\end{enumerate}
\end{Proposition}

\begin{proof}
(i) and (ii) can be verif\/ied by a straightforward calculation with the aid of Proposition~\ref{proposition2.1}(v). To prove (iii), note that the vector f\/ield $P_{k+1}dF$ is written as
\begin{gather*}
 P_{k+1}dF = \left(
\begin{array}{@{}ccc|ccc@{}}
 & & [X_0, \, \cdot \,] & & & \\
 & \rotatebox{80}{$\ddots$} & \vdots & & & \\
\,[X_0, \, \cdot \,] & \dots & [X_{k-1}, \, \cdot \,] & & & \\ \hline
 & & & -[X_{k+1}, \, \cdot \,] & \dots & -[X_n, \, \cdot \,] \\
 & & & \vdots & \rotatebox{80}{$\ddots$} & \\
 & & & -[X_n, \, \cdot \,] & & \\
\end{array}
\right) \left(
\begin{array}{@{}c@{}}
\nabla _1F \\
\vdots \\
\nabla _kF \\
\nabla _{k+1} F \\
\vdots \\
\nabla_n F
\end{array}
\right),
\end{gather*}
and similarly for $P_kdF$. Using them, write down the equation of $X_j$ for the vector f\/ield $(\lambda P_{k+1}-P_k)dF$. For example, the equation for $X_1$ is $dX_1/dt = \lambda [X_0, \nabla _kF]-[X_0,\nabla _{k-1}F]$. Summing up the equations of $\lambda ^{n-j}X_j$ proves the desired result.

(iv) Since $\lambda ^{m_in}\psi_{i,k} + \dots +\psi_{i,m_in+k}$ is the Casimir of $\lambda P_{k+1}-P_k$, we have
\begin{gather*}
(\lambda P_{k+1}-P_k) d(\lambda ^{m_in}\psi_{i,k} + \lambda ^{m_in-1}\psi_{i, k+1}+ \dots +\psi_{i,m_in+k})=0.
\end{gather*}
Expanding this yields the f\/irst equality. The second equality is conf\/irmed by a straightforward calculation.

(v) Due to Part (iv), we can assume that $k=n$. Because of the property $[P_{n+1}dF, P_{n+1}dG] = P_{n+1}d\{ G,F\}$ of a Poisson bracket (the left hand side is the Lie bracket for vector f\/ields), it is suf\/f\/icient to show the equality $\{ \psi_{i',j'}, \psi_{i,j}\}_{n+1} = 0$ for $i,i' = 1,\dots ,h$ and $j,j' = 1,\dots ,(m_i+1)n$. When $j=1,\dots ,n$, it is trivial because $\psi_{i,j}$ is the Casimir of $P_{n+1}$. Next, we have
\begin{gather*}
 \{ \lambda ^{m_in}\psi_{i',k}+\lambda ^{m_in-1}\psi_{i',k+1} +\dots + \psi_{i', m_in+k}, \psi_{i,n+j} \}_{n+1} \\
\qquad{} = \langle d(\lambda ^{m_in}\psi_{i',k}+\lambda ^{m_in-1}\psi_{i',k+1} +\dots + \psi_{i', m_in+k}) ,
 P_{n+1}d\psi _{i,n+j} \rangle \\
\qquad{} = \langle d(\lambda ^{m_in}\psi_{i',k}+\lambda ^{m_in-1}\psi_{i',k+1} +\dots + \psi_{i', m_in+k}) ,
 P_{k+1}d\psi _{i,k+j} \rangle \\
\qquad{} = \langle d(\lambda ^{m_in}\psi_{i',k}+\lambda ^{m_in-1}\psi_{i',k+1} +\dots + \psi_{i', m_in+k}) ,
 (\lambda P_{k+1} - P_k)dG_{i,k,j} \rangle \\
\qquad{} = -\langle dG_{i,k,j} , (\lambda P_{k+1}-P_k)d(\lambda ^{m_in}\psi_{i',k}+\lambda ^{m_in-1}\psi_{i',k+1} +\dots + \psi_{i', m_in+k}) \rangle =0.
\end{gather*}
This provides
\begin{gather*}
\{ \psi_{i',k}, \psi_{i, n+j}\}_{n+1} = \dots = \{ \psi_{i',m_in+k}, \psi_{i, n+j}\}_{n+1} = 0,
\end{gather*}
for any $k=1,\dots ,n$ and any $j=1,\dots ,m_in$, which completes the proof.
\end{proof}

\begin{Theorem}\label{theorem2.4}
Suppose that the constant $x_0$ for the definition of $\mathfrak{g}^0_n$ is chosen so that the functions $\{ \psi_{i,j}\}_{i,j}$ are functionally independent. Then, the vector field $P_{k+1}d\psi_{i, k+j}$, which is independent of~$k$, is completely integrable in the Liouville sense for any $i$ and~$j$.
\end{Theorem}

\begin{proof} Recall $\dim (\mathfrak{g}) = d$ and $\operatorname{rank}(\mathfrak{g}) = h$. Thus, $\dim (\mathfrak{g}^0_n) = nd$. Since $P_{k+1}$ has $nh$ Casimir functions, the dimension of a symplectic leaf $S$ of $P_{k+1}$ is $n(d-h)$. On the leaf $S$, the vector f\/ields $\{ P_{k+1}d\psi_{i, k+j} \}_{i,j}$ def\/ine $n(d-h)$-dimensional Hamiltonian systems, among which nonzero vector f\/ields are for $i=1,\dots ,h$ and $j=1, \dots ,m_in$.
Further, these nonzero vector f\/ields commute with each other and they are linearly independent due to the assumption. The number of the nonzero vector f\/ields is
\begin{gather*}
\sum^h_{i=1}m_in = \frac{1}{2}(\dim (\mathfrak{g}) - \operatorname{rank}(\mathfrak{g}))n = \frac{1}{2}n (d-h) = \frac{1}{2} \dim (S).
\end{gather*}
Hence, the Liouville theorem shows that the vector f\/ields are integrable.
\end{proof}

In what follows, we suppose the above assumption; the constant~$x_0$ for the def\/inition of $\mathfrak{g}^0_n$ is chosen so that the functions $\{ \psi_{i,j}\}_{i,j}$ are functionally independent. That is, the dif\/ferentials $\{ d\psi_{i,j}\}_{i,j}$ are linearly independent except for f\/inite points.

\subsection{Symplectic reduction}\label{section2.3}

The next purpose is to perform a symplectic reduction \cite{Mag1, Mag2, Mag3, Mag4}.

\begin{Theorem}\label{theorem2.5}
The $h$-dimensional distribution $D$ defined by
\begin{gather*}
D = \operatorname{span}\{ P_kd\psi_{i,k} \, | \, i=1,\dots ,h\}
\end{gather*}
is integrable in the Frobenius sense. The vector fields $P_k d\psi_{i,k}$ are linear for $i=1, \dots ,h$.
\end{Theorem}

\begin{proof}The f\/irst statement follows from Proposition~\ref{proposition2.3}(v). Since $P_kd\psi_{i,k}$ is independent of~$k$, we obtain $P_kd\psi_{i,k}=P_1d\psi_{i,1}$. Since $\operatorname{wdeg}(\psi_{i,1}) = 1$, $d\psi_{1,i}$ is a constant, while $P_1$ is linear in $(x_1,\dots ,x_n)$.
\end{proof}

The dif\/ferential equation for $P_kd\psi_{i,k} = P_1d\psi_{i,1}$ is given by
\begin{gather*}
\frac{d}{dt}X_\lambda = [X_\lambda , \nabla _1G_{i,1,1}] = [\nabla _1 \psi_{i,1}, X_\lambda ].
\end{gather*}
Since $\nabla _1 \psi_{i,1}$ is independent of $\lambda $, this is decomposed as
\begin{gather*}
\frac{d}{dt}X_k = [\nabla _1\psi_{i,1}, X_k], \qquad k=1,\dots ,n.
\end{gather*}
In coordinates, it is expressed as
\begin{gather}
\frac{dx_k}{dt} = -A_k \frac{\partial \psi_{i,1}}{\partial x_1}(x_1), \qquad k=1,\dots ,n.\label{D}
\end{gather}
Let us consider the orbit space $\pi \colon \mathfrak{g}^0_n \to \mathfrak{g}^0_n/D$, which is a smooth manifold if points on $\mathfrak{g}^0_n$ at which $\dim (D) < h$ is removed if necessary. The Marsden--Ratiu reduction theorem~\cite{Mag1} states that the orbit space $\mathfrak{g}^0_n/D$ is again a multi-Poisson manifold with compatible Poisson tensors denoted by $\widetilde{P}_1, \dots ,\widetilde{P}_{n+1}$. They are def\/ined by $\widetilde{P}_k = \pi_* P_k\pi^*$. Let $\{ \,\, , \,\, \}_k$ and $\{ \,\, ,\,\,\}_k'$ be Poisson brackets associated with~$P_k$ and~$\widetilde{P}_k$, respectively. For a function $F$ on $\mathfrak{g}^0_n$ which is constant along each integral manifold of~$D$, a function $\widetilde{F}$ on $\mathfrak{g}^0_n/D$ is well-def\/ined through $\widetilde{F} \circ \pi = \pi^* \widetilde{F} = F$. Conversely, for a function $\widetilde{F}$ on $\mathfrak{g}^0_n/D$, we can f\/ind a function $F$ on $\mathfrak{g}^0_n$, which is constant along~$D$, such that $\widetilde{F}\circ \pi = F$. Then, $\{ \,\, ,\,\,\}_k'$ is given by $\{ \widetilde{F}, \widetilde{G}\}_k' \circ \pi = \{ F, G\}_k$.

Because of Proposition~\ref{proposition2.3}(v), $\psi_{i,j}$ is constant along each integral manifold of $D$ and the projection $\widetilde{\psi}_{i,j}$ is well-def\/ined. The projected vector f\/ield is given by $\widetilde{P}_kd\widetilde{\psi}_{i,j} = \pi_* (P_kd\psi_{i,j})$.

It is convenient to realize $\mathfrak{g}^0_n/D$ as a submanifold of $\mathfrak{g}^0_n$. Let $\sigma \colon \mathfrak{g}^0_n/D \to \mathfrak{g}^0_n$ be a smooth section. The image $\sigma (\mathfrak{g}^0_n/D)$ is a~submanifold of $\mathfrak{g}^0_n$ which is dif\/feomorphic to $\mathfrak{g}^0_n/D$.
In Proposition~\ref{proposition2.6}(iii) below, $\mathfrak{g}^0_n/D$ is identif\/ied with a~submanifold in this manner.

\begin{Proposition}\label{proposition2.6}\quad
\begin{enumerate}\itemsep=0pt
\item[$(i)$] Casimir functions of $\widetilde{P}_{k+1}$ are $\widetilde{\psi}_{i,j}$, $i=1,\dots ,h$, for
$j=1,2,\dots ,k+1$ and for $j=m_in+k+1, m_in+k+2,\dots ,(m_i+1)n$.

\item[$(ii)$] Casimir functions of the combination $\lambda \widetilde{P}_{k+1} - \widetilde{P}_k$ are $\widetilde{\psi}_{i,j}$, $i=1,\dots ,h$, for $j=1,2,\dots ,k$ and for $j=m_in+k+1, m_in+k+2,\dots ,(m_i+1)n$, and
\begin{gather*}
\lambda ^{m_in-1}\widetilde{\psi}_{i,k+1} + \lambda ^{m_in-2}\widetilde{\psi}_{i, k+2}
 + \dots + \widetilde{\psi}_{i,m_in+k}, \qquad i=1,\dots ,h.
\end{gather*}
\item[$(iii)$] For a smooth function $\widetilde{F}\colon \mathfrak{g}^0_n/D \to \mathbb{C}$, there exist scalar-valued functions $\beta_1,\dots ,\beta_h \colon \mathfrak{g}^0_n/D$ $\to \mathbb{C}$ such that the equation for the vector field $(\lambda \widetilde{P}_{k+1} - \widetilde{P}_k)d\widetilde{F}$ is expressed in Lax form as
\begin{gather*}
\frac{d}{dt}\widetilde{X}_\lambda = \big[\widetilde{X}_\lambda , \nabla _k\widetilde{F}\big]
 - \sum^h_{i=1}\beta_i \big[\widetilde{X}_\lambda , \nabla _1 \widetilde{\psi}_{i,1}\big]
 = \left[\widetilde{X}_\lambda , \nabla _k\widetilde{F}-\sum^h_{i=1}\beta_i \nabla _1 \widetilde{\psi}_{i,1}\right],
\end{gather*}
where $\widetilde{X}_\lambda = X_\lambda |_{\mathfrak{g}^0_n/D}$ and $\nabla _k\widetilde{F} = (\nabla _kF)|_{\mathfrak{g}^0_n/D}$.

\item[$(iv)$] Define the function $\widetilde{G}_{i,k,j}$ to be
\begin{gather*}
\widetilde{G}_{i,k,j} = -\big( \lambda ^{j-1}\widetilde{\psi}_{i,k} + \lambda ^{j-2} \widetilde{\psi}_{i,k+1} + \dots + \widetilde{\psi}_{i,k+j-1}\big).
\end{gather*}
Then, the equality
\begin{gather*}
\widetilde{P}_{k+1}d\widetilde{\psi}_{i,k+j} = \widetilde{P_k}d\widetilde{\psi}_{i,k+j-1}
 = \big(\lambda \widetilde{P}_{k+1}-\widetilde{P}_k\big)d\widetilde{G}_{i,k,j}
\end{gather*}
holds for $i=1,\dots ,h$, $j=2,\dots ,m_in$ and $k=1,\dots ,n$.
\item[$(v)$] The vector fields $\widetilde{P}_{k+1}d\widetilde{\psi}_{i,k+j}$ for $i=1,\dots ,h$ and $j=2,\dots , m_in$
commute with each other $($note that it is zero when $j\notin \{ 2,\dots ,m_in\})$.
\end{enumerate}
\end{Proposition}

\begin{proof}
(i) $\widetilde{\psi}_{i,j}$ for $j=1,\dots ,k$ and $j=m_in+k+1, \dots , (m_i+1)n$ are Casimir of $\widetilde{P}_{k+1}$ because they are Casimir of $P_{k+1}$. For $\widetilde{\psi}_{i,k+1}$, we have
\begin{gather*}
\big\{ \widetilde{F}, \widetilde{\psi}_{i,k+1} \big\}'_{k+1}
 = \{ F, \psi_{i, k+1}\}_{k+1} = \langle dF , P_{k+1}d\psi_{i,k+1}\rangle = (P_k d\psi_{i,k})(F).
\end{gather*}
The right hand side becomes zero because $F$ is constant along $D$.

(ii) The f\/irst statement (on $\widetilde{\psi}_{i,j}$) is trivial because they are common Casimir of $\widetilde{P}_{k+1}$ and $\widetilde{P}_k$.
The last function $\lambda ^{m_in-1} \widetilde{\psi}_{i,k+1} + \cdots $ is a projection of the function given in Proposition~\ref{proposition2.3}(ii).

The results of (iv) and (v) are projections of those of Proposition~\ref{proposition2.3}(iv) and~(v).

To prove (iii), $\mathfrak{g}^0_n/D$ is identif\/ied with a submanifold of $\mathfrak{g}^0_n$ as above. Put $F = \pi^* \widetilde{F}$. We have to calculate the projection of the vector f\/ield $[X_\lambda , \nabla _kF]$ onto $\mathfrak{g}^0_n/D$ (see Proposition~\ref{proposition2.3}(iii)).
At f\/irst, we restrict the domain to $\mathfrak{g}^0_n/D$ as
\begin{gather*}
[X_\lambda , \nabla _kF]|_{\mathfrak{g}^0_n/D}
 = \big[X_\lambda |_{\mathfrak{g}^0_n/D}, (\nabla_k F)|_{\mathfrak{g}^0_n/D}\big]
 = \big[\widetilde{X}_\lambda , \nabla _k \widetilde{F}\big].
\end{gather*}
Since this is not tangent to $T(\mathfrak{g}^0_n/D)$, we calculate the projection of it according to the decomposition $T\mathfrak{g}^0_n = T(\mathfrak{g}^0_n/D) \oplus D$. Then, (iii) follows from the fact that the distribution $D$ is spanned by the vector f\/ields of the form $[X_\lambda , \nabla _1 \psi _{i,1}]$.
\end{proof}

\subsection{Isospectral deformation to isomonodromic deformation}\label{section2.4}

Now we have $(m_in-1)h$ distinct vector f\/ields on $\mathfrak{g}^0_n/D$
\begin{alignat*}{3}
&j=2\colon \quad && \widetilde{P}_1 d\widetilde{\psi}_{i,2} = \dots =
 \widetilde{P}_{k+1}d\widetilde{\psi}_{i,k+2}=\dots =\widetilde{P}_{n+1}d\widetilde{\psi}_{i,n+2}, & \\
&\cdots && \cdots\cdots\cdots\cdots\cdots\cdots\cdots\cdots\cdots\cdots\cdots\cdots\cdots\cdots\cdots\cdots &\\
&j=j\colon\quad && \widetilde{P}_1 d\widetilde{\psi}_{i,j} = \dots =
 \widetilde{P}_{k+1}d\widetilde{\psi}_{i,k+j}=\dots =\widetilde{P}_{n+1}d\widetilde{\psi}_{i,n+j}, & \\
&\cdots && \cdots\cdots\cdots\cdots\cdots\cdots\cdots\cdots\cdots\cdots\cdots\cdots\cdots\cdots\cdots\cdots &\\
&j=m_in\colon \quad && \widetilde{P}_1 d\widetilde{\psi}_{i,m_in} = \dots =
 \widetilde{P}_{k+1}d\widetilde{\psi}_{i,k+m_in}=\dots =\widetilde{P}_{n+1}d\widetilde{\psi}_{i,(m_i+1)n}.
\end{alignat*}
They are $(nd-h)$-dimensional integrable systems. For f\/ixed $k$, a symplectic leaf of the Poisson structure $\widetilde{P}_k$ is given by a level surface of Casimir functions as
\begin{gather*}
S_k = \big\{ \widetilde{\psi}_{i,j} = \alpha _{i,j}\, (\text{const}) \, | \, i=1,\dots ,h,\, j=1,\dots ,k,\, j=m_in+k,\dots ,(m_i+1)n\big\}.
\end{gather*}
Restricted on the symplectic leaf, the vector f\/ields become $(nd-nh-2h)$-dimensional completely integrable Hamiltonian systems of the form
\begin{alignat}{3}
& \widetilde{P}_{k+1}d\widetilde{\psi}_{i,k+j}\colon \quad & & \frac{d}{dt}\widetilde{X}_\lambda
 = \big[\widetilde{X}_\lambda , A_\lambda \big] \qquad \text{on $S_k$},& \label{2-7} \\
& && A_\lambda :=\nabla _k\widetilde{G}_{i,k,j}-\sum^h_{i=1}\beta_i \nabla _1 \widetilde{\psi}_{i,1}. & \nonumber
\end{alignat}
Both of $\widetilde{X}_\lambda $ and $A_\lambda $ depend on parameters $\{ \alpha _{i,j}\}_{i,j}$ which def\/ine the symplectic leaf. Thus, we write $\widetilde{X}_\lambda $ as $\widetilde{X}_\lambda (t, \alpha )$, where $\alpha $ denotes the collection of parameters~$\alpha _{i,j}$.

Now suppose that there exists a parameter $\alpha _{i',j'}$ such that the following condition holds
\begin{gather}
\frac{\partial \widetilde{X}_\lambda }{\partial \alpha _{i',j'}}(t, \alpha ) = \lambda ^l \frac{\partial A_\lambda }{\partial \lambda}\label{2-8}
\end{gather}
for some integer $l$. Equation~(\ref{2-7}) is put together with equation~(\ref{2-8}) to yield
\begin{gather*}
\frac{\partial \widetilde{X}_\lambda}{\partial t}(t,\alpha ) + \frac{\partial \widetilde{X}_\lambda }{\partial \alpha _{i',j'}}(t, \alpha )
 = \big[\widetilde{X}_\lambda , A_\lambda \big]+\lambda ^l \frac{\partial A_\lambda }{\partial \lambda }.
\end{gather*}
Def\/ine the Lax matrix $L_\lambda $ by
\begin{gather*}
L_\lambda := \frac{1}{\lambda ^l} \widetilde{X}_\lambda (t, \alpha )|_{\alpha _{i',j'} = t},
\end{gather*}
where the parameter $\alpha _{i',j'}$ satisfying the condition (\ref{2-8}) is replaced by~$t$. Then, the above equation is rewritten as
\begin{gather}
\frac{dL_\lambda}{dt} = [L_\lambda , A_\lambda ]+\frac{\partial A_\lambda }{\partial \lambda },
\label{2-9}
\end{gather}
which is known as the isomonodromic deformation equation. It is known that a system written as the isomonodromic deformation equation enjoys the Painlev\'{e} property. The function $\widetilde{\psi}_{i,k+j}$ restricted on~$S_k$ will be a Hamiltonian function of the Painlev\'{e} equation after replacing $\alpha _{i',j'} \mapsto t$ and changing to Darboux's coordinates if necessary.

For ABCD-type simple Lie algebras, the dimensions of Painlev\'{e} systems obtained in this manner are summarized in Table~\ref{table3}. In particular, the dimension is $2$ when $\mathfrak{g} = \mathfrak{sl}_2$, $n=2$, and it is $4$ when $\mathfrak{g} = \mathfrak{sl}_2$, $n=3$ or $\mathfrak{g} = \mathfrak{so}_5$, $n=1$. From the next sections, we will demonstrate our method for these cases. In particular, the Hamiltonian functions shown in Section~\ref{section1} will be obtained.

\begin{table}[h]\centering\caption{The dimensions of Painlev\'{e} systems.}\label{table3}\vspace{1mm}
\begin{tabular}{|c|c|c|c|c|c|} \hline
& & dimension & $h=1$ & $h=2$ & $h=3$\\ \hline \hline
$\text{A}_h (h\geq 1)$ & $\mathfrak{sl}_{h+1}$ & $nh^2+nh-2h$ & $2n-2$ & $6n-4$ & $12n-6$ \\ \hline
$\text{B}_h (h\geq 2)$ & $\mathfrak{so}_{2h+1}$& $2nh^2-2h$ & $-$ & $8n-4$ & $18n-6$\\ \hline
$\text{C}_h (h\geq 3)$ & $\mathfrak{sp}_{2h}$ & $2nh^2-2h$ & $-$ & $-$ & $18n-6$ \\ \hline
$\text{D}_h (h\geq 4)$ & $\mathfrak{so}_{2h}$ & $2nh^2-2nh-2h$ & $-$ & $-$ & $-$ \\ \hline
\end{tabular}
\end{table}

\section[2-dimensional Painlev\'{e} equations: $\mathfrak{g} = \mathfrak{sl}_2$, $n=2$]{2-dimensional Painlev\'{e} equations: $\boldsymbol{\mathfrak{g} = \mathfrak{sl}_2}$, $\boldsymbol{n=2}$}\label{section3}

In this case, a general element of $\mathfrak{g}_n$ is written by
\begin{gather*}
X_\lambda = \lambda ^2 \left(
\begin{matrix}
u_0 & v_0 \\
w_0 & -u_0
\end{matrix}
\right) + \lambda \left(
\begin{matrix}
u_1 & v_1 \\
w_1 & -u_1
\end{matrix}
\right) + \left(
\begin{matrix}
u_2 & v_2 \\
w_2 & -u_2
\end{matrix}
\right).
\end{gather*}
The Painlev\'{e} equation obtained by our method depends on a choice of $x_0 = (u_0, v_0, w_0)$. We consider the following two cases.
\begin{gather*}
\text{(I)}\ \left(
\begin{matrix}
u_0 & v_0 \\
w_0 & -u_0
\end{matrix}
\right) = \left(
\begin{matrix}
1 & 0 \\
0 & -1
\end{matrix}
\right), \qquad \text{(II)}\ \left(
\begin{matrix}
u_0 & v_0 \\
w_0 & -u_0
\end{matrix}
\right) = \left(
\begin{matrix}
0 & 0 \\
1 & 0
\end{matrix}
\right).
\end{gather*}
From the former case, we will obtain the second and fourth Painlev\'{e} equations $\text{P}_\text{II}$, $\text{P}_\text{IV}$, and from the latter one, we will obtain the f\/irst and second Painlev\'{e} equations $\text{P}_\text{I}$, $\text{P}_\text{II}$.

\subsection{Case (I)}\label{section3.1}

\looseness=-1 In this case, the functions $\psi _{i,j} = \psi _j$ (since $h=\operatorname{rank}(\mathfrak{g}) = 1$, we omit the subscript~$i$) are given by
\begin{alignat*}{3}
& \psi _1 = 2u_1 , \qquad && \psi _2 = 2u_2 + u_1^2 + v_1w_1, & \\
& \psi _3 = 2u_1u_2 + v_2w_1 + v_1w_2, \qquad && \psi _4 = u_2^2+v_2w_2.&
\end{alignat*}
(see Example~\ref{example2.2}). The dif\/ferential equation~(\ref{D}) def\/ining the distribution $D$ is $u_j' = 0$, $v_j' = 2v_j$, $w_j' = -2w_j$ for $j=1,2$. This is solved as a function of $w_1$ as
\begin{gather*}
u_1 = U_1, \qquad u_2 = U_2,\qquad v_1 = V_2/w_1,\qquad v_2 = V_3/w_1,\qquad w_2 = W_1 w_1,
\end{gather*}
where $U_1$, $U_2$, $V_2$, $V_3$, $W_1$ are integral constants (initial values at $w_1=1$), for which the subscripts are given so that they are consistent with the weighted degrees (for example, since $\operatorname{wdeg}(v_2w_1) = 2+1 = 3$, the weighted degree of $V_3$ is three). This relation def\/ines a coordinate transformation
\begin{gather*}
(u_1, v_1, w_1, u_2, v_2, w_2) \mapsto (U_1, V_2, w_1, U_2, V_3, W_1).
\end{gather*}
In the new coordinates, integral manifolds of the distribution $D$ are straight lines along $w_1$-axis. In particular, the subset $\{ w_1 = 1\} \subset \mathfrak{g}^0_n$ gives the realization of the orbit space $\mathfrak{g}^0_n/D$ as a~submanifold and $(U_1, V_2, U_2, V_3, W_1)$ provides a global coordinate system of $\mathfrak{g}^0_n/D$.

At this stage, we have on $\mathfrak{g}^0_n/D$
\begin{alignat*}{3}
& \widetilde{\psi}_1 = 2U_1, \qquad && \widetilde{\psi}_2 = 2U_2+U_1^2+V_2, & \\
& \widetilde{\psi}_3 = 2U_1U_2 + V_3 + V_2W_1, \qquad && \widetilde{\psi}_4 = U_2^2 + V_3W_1,&
\end{alignat*}
and three Poisson structures $\widetilde{P}_1$~(Casimirs are $\widetilde{\psi}_1$, $\widetilde{\psi}_3$, $\widetilde{\psi}_4$),
$\widetilde{P}_2$ (Casimirs are $\widetilde{\psi}_1$, $\widetilde{\psi}_2$, $\widetilde{\psi}_4$), $\widetilde{P}_3$~(Casimirs are $\widetilde{\psi}_1$, $\widetilde{\psi}_2$, $\widetilde{\psi}_3$), and vector f\/ields $\widetilde{P}_3d\widetilde{\psi}_4 = \widetilde{P}_2 d\widetilde{\psi}_3 = \widetilde{P}_1d\widetilde{\psi}_2$ which are expressed as the Lax equation $d\widetilde{X}_\lambda /dt = [A_\lambda, \widetilde{X}_\lambda ]$, where
\begin{gather*}
\widetilde{X}_\lambda = \lambda ^2 \left(
\begin{matrix}
1 & 0 \\
0 & -1
\end{matrix}
\right) + \lambda \left(
\begin{matrix}
U_1 & V_2 \\
1 & -U_1
\end{matrix}
\right) + \left(
\begin{matrix}
U_2 & V_3 \\
W_1 & -U_2
\end{matrix}
\right), \\
A_\lambda = \lambda \left(
\begin{matrix}
1 & 0 \\
0 & -1
\end{matrix}
\right) + \left(
\begin{matrix}
U_1 & V_2 \\
1 & -U_1
\end{matrix}
\right) - \left(
\begin{matrix}
W_1 & 0 \\
0 & -W_1
\end{matrix}
\right).
\end{gather*}

The next purpose is to restrict the vector f\/ields on a symplectic leaf. We will consider $\widetilde{P}_3d\widetilde{\psi}_4$ and $\widetilde{P}_2 d\widetilde{\psi}_3$ separately ($\widetilde{P}_1d\widetilde{\psi}_2$ will not be considered because there are no parameters satisfying~(\ref{2-8})).

(i) Consider the vector f\/ield $\widetilde{P}_3 d\widetilde{\psi}_4$. For the Poisson tensor $\widetilde{P}_3$, a symplectic leaf is def\/ined by the level surface $\{ \widetilde{\psi}_j = \text{const},\, j=1,2,3\}$. In order for the condition
\begin{gather}
\frac{\partial \widetilde{X}_\lambda }{\partial \alpha } = \frac{\partial A_\lambda }{\partial \lambda } = \left(
\begin{matrix}
1 & 0 \\
0 & -1
\end{matrix}
\right) \label{3-1}
\end{gather}
to be satisf\/ied, we f\/ind that $U_2$ in $\widetilde{X}_\lambda $ has to include a parameter $\alpha $ which will be replaced by $t$ later.
For this purpose, we take the symplectic leaf
\begin{gather*}
S = \big\{ 2U_1 = 0,\, 2U_2+U_1^2+V_2 = 2\alpha _2 ,\, 2U_1U_2 + V_3 + V_2W_1 = \alpha _3 \big\}.
\end{gather*}
Hence, we put $U_1 = 0$, $U_2 = \alpha _2 - V_2/2$, $V_3 = \alpha _3 - V_2W_1$, and $(V_2, W_1)$ gives a global coordinate system for the symplectic leaf. Then, it turns out that $\widetilde{X}_\lambda $ satisf\/ies the condition (\ref{3-1}) with $\alpha =\alpha _2$ on the symplectic leaf. Finally, by replacing $\alpha _2$ by $t$, we obtain the isomonodromic deformation equation (\ref{2-9}) with
\begin{gather*}
\begin{split}
& L_\lambda = \lambda ^2 \left(
\begin{matrix}
1 & 0 \\
0 & -1
\end{matrix}
\right) + \lambda \left(
\begin{matrix}
0 & V_2 \\
1 & 0
\end{matrix}
\right) + \left(
\begin{matrix}
t-V_2/2 & \alpha _3 - V_2W_1 \\
W_1 & -(t-V_2/2)
\end{matrix}
\right), \\
& A_\lambda = \lambda \left(
\begin{matrix}
1 & 0 \\
0 & -1
\end{matrix}
\right) + \left(
\begin{matrix}
-W_1 & V_2 \\
1 & W_1
\end{matrix}
\right).
\end{split}
\end{gather*}
The Poisson tensor $\widetilde{P}_3$ on the symplectic leaf with coordinates $(V_2,W_1)$ is given by
\begin{gather*}
\widetilde{P}_3 = \left(
\begin{matrix}
0 & 2 \\
-2 & 0
\end{matrix}
\right).
\end{gather*}
To change to Darboux's coordinates, we put $V_2 = -2p_2$, $W_1 = q_1$. Then, $\widetilde{P}_3$ becomes the canonical symplectic matrix. In the coordinates $(q_1, p_2)$, the isomonodromic deformation equation is a Hamiltonian system. The Hamiltonian function $\widetilde{\psi}_4$ for the vector f\/ield $\widetilde{P}_3d\widetilde{\psi}_4$ is written as
\begin{gather*}
\widetilde{\psi}_4 = U_2^2 + V_3W_1 = (t - V_2/2)^2 + (\alpha _3 - V_2W_1)W_1 \\
\hphantom{\widetilde{\psi}_4}{} = p_2^2 + 2q_1^2p_2 + 2tp_2 + \alpha _3 q_1 + t^2.
\end{gather*}
This is reduced to the Hamiltonian function (\ref{2dimP2}) of the second Painlev\'{e} equation by a certain coordinate change and the isomonodromic deformation equation~(\ref{2-9}) is equivalent to the second Painlev\'{e} equation.

(ii) Consider the vector f\/ield $\widetilde{P}_2 d\widetilde{\psi}_3$. For the Poisson tensor $\widetilde{P}_2$, a symplectic leaf is def\/ined by the level surface $\{ \widetilde{\psi}_j = \text{const},\, j=1,2,4\}$.
In order for the condition
\begin{gather}
\frac{\partial \widetilde{X}_\lambda }{\partial \alpha } = \lambda \frac{\partial A_\lambda }{\partial \lambda } = \lambda \left(
\begin{matrix}
1 & 0 \\
0 & -1
\end{matrix}
\right)\label{3-2}
\end{gather}
to be satisf\/ied, we f\/ind that $U_1$ in $\widetilde{X}_\lambda $ has to include a parameter $\alpha $ which will be replaced by $t$ later, and the other components of $\widetilde{X}_\lambda $ cannot include $\alpha $. For this purpose, we take the symplectic leaf
\begin{gather*}
S = \big\{ 2U_1 = 2 \alpha _1,\, 2U_2+U_1^2+V_2 = \alpha _2+\alpha _1^2 ,\, U_2^2+V_3W_1 = \alpha _4\big \}.
\end{gather*}
This relation is rewritten as
\begin{gather*}
U_1 = \alpha _1,\qquad V_2 = \alpha _2 - 2U_2, \qquad V_3 = \big(\alpha _4-U_2^2\big)/W_1.
\end{gather*}
By substituting them, $\widetilde{X}_\lambda $ satisf\/ies the condition (\ref{3-2}) with $\alpha =\alpha _1$. Finally, by replacing $\alpha _1$ by~$t$, we obtain the isomonodromic deformation equation~(\ref{2-9}).

The Poisson tensor $\widetilde{P}_2$ on the symplectic leaf with coordinates $(U_2,W_1)$ is given by
\begin{gather*}
\widetilde{P}_2 = \left(
\begin{matrix}
0 & W_1 \\
-W_1 & 0
\end{matrix}
\right).
\end{gather*}
For Darboux's coordinates, we put $U_2 = p_1q_1 - \beta_2$ and $W_1 = q_1$, where $\beta_2$ is an arbitrary constant. Then, $\widetilde{P}_2$ is transformed to the canonical symplectic matrix. In the coordinates $(q_1, p_1)$, the isomonodromic deformation equation is a~Hamiltonian system. The Hamiltonian function $\widetilde{\psi}_3$ for the vector f\/ield $\widetilde{P}_2d\widetilde{\psi}_3$ is written as
\begin{gather*}
\widetilde{\psi}_3 = 2U_1U_2 + V_3 + V_2W_1 = 2t U_2 + \big(\alpha _4-U_2^2\big)/W_1 + (\alpha _2 - 2U_2)W_1 \\
\hphantom{\widetilde{\psi}_3}{} = -p_1^2 q_1 - 2p_1q_1^2 + 2t p_1q_1 + 2\beta_2 p_1 + (\alpha _2 + 2\beta_2) q_1 - 2\beta_2 t + \frac{\alpha _4-\beta_2^2}{q_1}.
\end{gather*}
We choose the free parameter $\beta_2$ to be $\alpha _4 = \beta_2^2$ so that the Hamiltonian becomes a polynomial.
This is the Hamiltonian function (\ref{2dimP4}) of the fourth Painlev\'{e} equation up to some scaling.
The isomonodromic deformation equation (\ref{2-9}) is equivalent to the fourth Painlev\'{e} equation, where
\begin{gather*}
L_\lambda = \lambda \left(
\begin{matrix}
1 & 0 \\
0 & -1
\end{matrix}
\right) + \left(
\begin{matrix}
t & \alpha _2+2\beta_2 - 2p_1q_1 \\
1 & -t
\end{matrix}
\right) + \frac{1}{\lambda }\left(
\begin{matrix}
p_1q_1-\beta_2 & -p_1^2q_1+2\beta_2 p_1 \\
q_1 & -(p_1q_1-\beta_2)
\end{matrix}
\right), \\
A_\lambda = \lambda \left(
\begin{matrix}
1 & 0 \\
0 & -1
\end{matrix}
\right) + \left(
\begin{matrix}
t-q_1 & \alpha _2+2\beta_2 - 2p_1q_1 \\
1 & -(t-q_1)
\end{matrix}
\right).
\end{gather*}

\subsection{Case (II)}\label{section3.2}

In this case, the functions $\psi _{i,j} = \psi _j$ are given by
\begin{alignat*}{3}
& \psi _1 = v_1, \qquad && \psi _2 = u_1^2 + v_2 + v_1w_1, & \\
& \psi _3 = 2u_1u_2 + v_2w_1 + v_1w_2, \qquad && \psi _4 = u_2^2+v_2w_2.&
\end{alignat*}
The dif\/ferential equation (\ref{D}) def\/ining the distribution $D$ is $u_j' = -v_j$, $v_j' = 0$, $w_j' = 2u_j$ for $j=1,2$. We can assume without loss of generality that $v_1 = 1$ by a suitable scaling of variables (indeed, $v_1$ is a common Casimir of $P_1$, $P_2$, $P_3$). Thus, the equations are solved as a function of~$u_1$ as
\begin{gather*}
v_1 = 1,\qquad v_2 = V_2,\qquad u_2 = V_2u_1 + U_3,\\ w_1 = -u_1^2 + W_2,\qquad w_2 = -V_2u_1^2 - 2U_3u_1 + W_4,
\end{gather*}
where $V_2$, $U_3$, $W_2$, $W_4$ are integral constants (initial values at $u_1=0$). This relation def\/ines a~coordinate transformation
\begin{gather*}
(u_1, w_1, u_2, v_2, w_2) \mapsto (u_1, W_2, U_3, V_2, W_4).
\end{gather*}
In the new coordinates, integral manifolds of the distribution $D$ are straight lines along $u_1$-axis. In particular, the subset $\{ u_1 = 0\} \subset \mathfrak{g}^0_n$ gives the realization of the orbit space $\mathfrak{g}^0_n/D$ and $(W_2, U_3, V_2, W_4)$ provides a global coordinate system of $\mathfrak{g}^0_n/D$ restricted to $v_1 = 1$.

On $\mathfrak{g}^0_n/D$, we have functions
\begin{gather*}
\widetilde{\psi}_1 = 1, \qquad \widetilde{\psi}_2 = V_2 + W_2 , \qquad
\widetilde{\psi}_3 = V_2W_2 + W_4, \qquad \widetilde{\psi}_4 = U_3^2 + V_2W_4,
\end{gather*}
and three Poisson structures $\widetilde{P}_1$ (Casimirs are $\widetilde{\psi}_1$, $\widetilde{\psi}_3$, $\widetilde{\psi}_4$), $\widetilde{P}_2$~(Casimirs are $\widetilde{\psi}_1$, $\widetilde{\psi}_2$, $\widetilde{\psi}_4$), $\widetilde{P}_3$~(Casimirs are $\widetilde{\psi}_1$, $\widetilde{\psi}_2$, $\widetilde{\psi}_3$), and vector f\/ields $\widetilde{P}_3d\widetilde{\psi}_4 = \widetilde{P}_2 d\widetilde{\psi}_3 = \widetilde{P}_1d\widetilde{\psi}_2$
which are expressed as the Lax equation $d\widetilde{X}_\lambda /dt = [A_\lambda, \widetilde{X}_\lambda ]$, where
\begin{gather*}
\widetilde{X}_\lambda = \lambda ^2 \left(
\begin{matrix}
0 & 0 \\
1 & 0
\end{matrix}
\right) + \lambda \left(
\begin{matrix}
0 & 1 \\
W_2 & 0
\end{matrix}
\right) + \left(
\begin{matrix}
U_3 & V_2 \\
W_4 & -U_3
\end{matrix}
\right), \\
A_\lambda = \lambda \left(
\begin{matrix}
0 & 0 \\
1 & 0
\end{matrix}
\right) + \left(
\begin{matrix}
0 & 1 \\
W_2 & 0
\end{matrix}
\right) - \left(
\begin{matrix}
0 & 0 \\
V_2 & 0
\end{matrix}
\right).
\end{gather*}
The next purpose is to restrict the vector f\/ields on a symplectic leaf.

(i) Consider the vector f\/ield $\widetilde{P}_3 d\widetilde{\psi}_4$. For the Poisson tensor $\widetilde{P}_3$, a symplectic leaf is def\/ined by the level surface $\{ \widetilde{\psi}_j = \text{const},\, j=1,2,3\}$. For the condition
\begin{gather}
\frac{\partial \widetilde{X}_\lambda }{\partial \alpha } = \frac{\partial A_\lambda }{\partial \lambda } = \left(
\begin{matrix}
0 & 0 \\
1 & 0
\end{matrix}
\right),
\label{3-3}
\end{gather}
we \looseness=-1 f\/ind that $W_4$ in $\widetilde{X}_\lambda $ has to include a parameter $\alpha $. To this end, we take the symplectic leaf as
\begin{gather*}
S = \big\{ \widetilde{\psi}_1 = 1,\, \widetilde{\psi}_2 = V_2 + W_2 = 0 ,\, \widetilde{\psi}_3 = V_2W_2 + W_4 = \alpha _4 \big\}.
\end{gather*}
Hence, we put $V_2 = -W_2$, $W_4 = \alpha _4 + W_2^2$, so that $(W_2, U_3)$ gives a global coordinate system on the leaf. Then, $\widetilde{X}_\lambda $ satisf\/ies the condition~(\ref{3-3}) with $\alpha =\alpha _4$. Finally, by replacing~$\alpha _4$ by~$t$, we obtain the isomonodromic deformation equation~(\ref{2-9}) with
\begin{gather*}
\begin{split}
& L_\lambda = \lambda ^2 \left(
\begin{matrix}
0 & 0 \\
1 & 0
\end{matrix}
\right) + \lambda \left(
\begin{matrix}
0 & 1 \\
W_2 & 0
\end{matrix}
\right) + \left(
\begin{matrix}
U_3 & -W_2 \\
t + W_2^2 & -U_3
\end{matrix}
\right), \\
& A_\lambda = \lambda \left(
\begin{matrix}
0 & 0 \\
1 & 0
\end{matrix}
\right) + \left(
\begin{matrix}
0 & 1 \\
2W_2 & 0
\end{matrix}
\right) .
\end{split}
\end{gather*}
The Poisson tensor $\widetilde{P}_3$ on the symplectic leaf is given by
\begin{gather*}
\widetilde{P}_3 = \left(
\begin{matrix}
0 & 1 \\
-1 & 0
\end{matrix}
\right),
\end{gather*}
which is already in canonical form. On the symplectic leaf, the function $\widetilde{\psi}_4$ is written as
\begin{gather*}
\widetilde{\psi}_4 = U_3^2 + V_2W_4 = U_3^2 - W_2^3 - tW_2.
\end{gather*}
This is the Hamiltonian function (\ref{2dimP1}) of the f\/irst Painlev\'{e} equation (up to some scaling) and the isomonodromic deformation equation (\ref{2-9}) coincides with the f\/irst Painlev\'{e} equation.

(ii) For the vector f\/ield $\widetilde{P}_2 d\widetilde{\psi}_3$, we again obtain the second Painlev\'{e} equation and the detailed calculation is omitted.

\section[4-dimensional Painlev\'{e} equations: $\mathfrak{g} = \mathfrak{sl}_2$, $n=3$]{4-dimensional Painlev\'{e} equations: $\boldsymbol{\mathfrak{g} = \mathfrak{sl}_2}$, $\boldsymbol{n=3}$}\label{section4}

In this case, a general element of $\mathfrak{g}_n$ is written as
\begin{gather*}
X_\lambda = \lambda ^3 \left(
\begin{matrix}
u_0 & v_0 \\
w_0 & -u_0
\end{matrix}
\right) + \lambda^2 \left(
\begin{matrix}
u_1 & v_1 \\
w_1 & -u_1
\end{matrix}
\right) + \lambda \left(
\begin{matrix}
u_2 & v_2 \\
w_2 & -u_2
\end{matrix}
\right) + \left(
\begin{matrix}
u_3 & v_3 \\
w_3 & -u_3
\end{matrix}
\right).
\end{gather*}
For the def\/inition of $\mathfrak{g}^0_n$, we again consider the following two cases.
\begin{gather*}
\text{(I)}\ \left(
\begin{matrix}
u_0 & v_0 \\
w_0 & -u_0
\end{matrix}
\right) = \left(
\begin{matrix}
1 & 0 \\
0 & -1
\end{matrix}
\right), \qquad \text{(II)}\ \left(
\begin{matrix}
u_0 & v_0 \\
w_0 & -u_0
\end{matrix}
\right) = \left(
\begin{matrix}
0 & 0 \\
1 & 0
\end{matrix}
\right).
\end{gather*}
From the former case, we will obtain Hamiltonian functions (\ref{4dimP22}), (\ref{4dimP4}), (\ref{1120}), and from the latter one, we will obtain (\ref{4dimP1}), (\ref{4dimP21}), (\ref{-1412}).

\subsection{Case (I)}\label{section4.1}

\looseness=-1 In this case, the functions $\psi _{i,j} = \psi _j$ (since $h=\operatorname{rank}(\mathfrak{g}) = 1$, we omit the subscript $i$) are given by
\begin{gather*}
\psi _1 = 2u_1, \\
\psi _2 = 2u_2 + u_1^2 + v_1w_1, \\
\psi _3 = 2u_1u_2 + 2u_3 + v_2w_1 + v_1w_2,\\
\psi _4 = u_2^2 + 2u_1u_3 + v_3w_1 + v_2w_2 + v_1w_3, \\
\psi_5 = 2u_2u_3+v_3w_2+v_2w_3, \\
\psi_6 = u_3^2 + v_3w_3.
\end{gather*}
The dif\/ferential equation (\ref{D}) def\/ining the distribution $D$ is $u_j' = 0$, $ v_j' = 2v_j$, $ w_j' = -2w_j$ for $j=1,2,3$. This is solved as a function of $w_1$ as
\begin{gather*}
 u_1 = U_1, \qquad u_2 = U_2,\qquad u_3= U_3, \\
 v_1 = V_2/w_1,\qquad v_2 = V_3/w_1,\qquad v_3 = V_4/w_1,\qquad w_2 = W_1 w_1,\qquad w_3 = W_2w_1,
\end{gather*}
where $U_1$, $U_2$, $ U_3$, $V_2$, $V_3$, $V_4$, $W_1$, $W_2$ are integral constants (initial values at $w_1 = 1$). This relation def\/ines a coordinate transformation
\begin{gather*}
(u_1, v_1, w_1, u_2, v_2, w_2, u_3, v_3, w_3) \mapsto (U_1, V_2, w_1, U_2, V_3, W_1, U_3, V_4, W_2).
\end{gather*}
In the new coordinates, integral manifolds of the distribution $D$ are straight lines along $w_1$-axis. In particular, the subset $\{ w_1 = 1\} \subset \mathfrak{g}^0_n$ gives the realization of the orbit space $\mathfrak{g}^0_n/D$ as a~submanifold and $(U_1, V_2, U_2, V_3, W_1, U_3, V_4, W_2)$ provides a global coordinate system of $\mathfrak{g}^0_n/D$.

At this stage, we have on $\mathfrak{g}^0_n/D$
\begin{gather*}
\widetilde{\psi}_1 = 2U_1, \\
\widetilde{\psi}_2 = 2U_2+U_1^2+V_2, \\
\widetilde{\psi}_3 = 2U_1U_2 + 2U_3 + V_3 + V_2W_1, \\
\widetilde{\psi}_4 = 2U_1U_3 + U_2^2 + V_4 + V_2W_2+ V_3W_1, \\
\widetilde{\psi}_5 = 2U_2U_3 + V_4W_1 + V_3W_2, \\
\widetilde{\psi}_6 = U_3^2 + V_4W_2,
\end{gather*}
and two vector f\/ields
\begin{gather}
 \widetilde{P}_2d\widetilde{\psi}_3 = \widetilde{P}_3d\widetilde{\psi}_4 = \widetilde{P}_4d\widetilde{\psi}_5,\label{vec1} \\
 \widetilde{P}_2d\widetilde{\psi}_4 = \widetilde{P}_3d\widetilde{\psi}_5 = \widetilde{P}_4d\widetilde{\psi}_6.\label{vec2}
\end{gather}
The dif\/ferential equations of these vector f\/ields are expressed in Lax form as
\begin{gather}
\frac{\partial \widetilde{X}_\lambda }{\partial t_1} = \big[A_1, \widetilde{X}_\lambda \big], \qquad
\frac{\partial \widetilde{X}_\lambda }{\partial t_2} = \big[A_2, \widetilde{X}_\lambda \big],\label{Lax4}
\end{gather}
respectively, where
\begin{gather*}
\widetilde{X}_\lambda = \lambda ^3 \left(
\begin{matrix}
1 & 0 \\
0 & -1
\end{matrix}
\right) + \lambda^2 \left(
\begin{matrix}
U_1 & V_2 \\
1 & -U_1
\end{matrix}
\right) + \lambda \left(
\begin{matrix}
U_2 & V_3 \\
W_1 & -U_2
\end{matrix}
\right) + \left(
\begin{matrix}
U_3 & V_4 \\
W_2 & -U_3
\end{matrix}
\right), \\
A_1 = \lambda \left(
\begin{matrix}
1 & 0 \\
0 & -1
\end{matrix}
\right) + \left(
\begin{matrix}
U_1 & V_2 \\
1 & -U_1
\end{matrix}
\right) - \left(
\begin{matrix}
W_1 & 0 \\
0 & -W_1
\end{matrix}
\right), \\
A_2 = \lambda^2 \left(
\begin{matrix}
1 & 0 \\
0 & -1
\end{matrix}
\right) + \lambda \left(
\begin{matrix}
U_1 & V_2 \\
1 & -U_1
\end{matrix}
\right) + \left(
\begin{matrix}
U_2 & V_3 \\
W_1 & -U_2
\end{matrix}
\right)
 - \left(
\begin{matrix}
W_2 & 0 \\
0 & -W_2
\end{matrix}
\right).
\end{gather*}

The next purpose is to restrict the vector f\/ields on a symplectic leaf.

(i) Consider the pair of vector f\/ields $\widetilde{P}_4 d\widetilde{\psi}_5$ and $\widetilde{P}_4 d\widetilde{\psi}_6$. For the Poisson tensor $\widetilde{P}_4$, a symplectic leaf is def\/ined by the level surface
$\{ \widetilde{\psi}_j = \text{const},\, j=1,2,3,4\}$.
In order for the two conditions
\begin{gather}
\frac{\partial \widetilde{X}_\lambda }{\partial \alpha } = \frac{\partial A_1 }{\partial \lambda } = \left(
\begin{matrix}
1 & 0 \\
0 & -1
\end{matrix}
\right), \qquad
\frac{\partial \widetilde{X}_\lambda }{\partial \alpha' } = \frac{\partial A_2 }{\partial \lambda } = 2\lambda \left(
\begin{matrix}
1 & 0 \\
0 & -1
\end{matrix}
\right) + \left(
\begin{matrix}
U_1 & V_2 \\
1 & -U_1
\end{matrix}
\right)
\label{4-1}
\end{gather}
to be satisf\/ied, we f\/ind that $U_3$ in $\widetilde{X}_\lambda $ has to include a parameter~$\alpha $, which will be replaced by~$t_1$ later,
and $U_2$ and $W_2$ have to include a parameter~$\alpha' $, which will be replaced by~$t_2$ later. For this purpose, we take the symplectic leaf
\begin{gather*}
S = \big\{ \widetilde{\psi}_1 = 0,\, \widetilde{\psi}_2 = 4\alpha _2 ,\, \widetilde{\psi}_3 = 2\alpha _3,\,
\widetilde{\psi}_4 = \alpha _4 + 4\alpha _2^2 \big\}.
\end{gather*}
Further, we change the coordinate as $W_2 = \widetilde{W}_2 + \alpha _2$ because $W_2$ should include a parameter. Then, the above relation for $S$ is rearranged as
\begin{gather*}
 U_1 = 0,\qquad U_2 = 2\alpha _2-\frac{1}{2}V_2,\qquad U_3 = \alpha _3-\frac{1}{2}V_3-\frac{1}{2}V_2W_1, \\
 V_4 = \alpha _4 + \alpha _2V_2 - \frac{1}{4}V_2^2 - V_2\widetilde{W}_2 - V_3W_1.
\end{gather*}
Substituting them into $\widetilde{X}_\lambda$, $A_1$ and $A_2$, we can verify the condition~(\ref{4-1}) with $\alpha = \alpha _3$ and $\alpha ' = \alpha _2$. By replacing $\alpha _3$, $\alpha _2$ by $t_1$, $t_2$, respectively, we obtain the isomonodromic deformation equations
\begin{gather}
\frac{\partial L_\lambda }{\partial t_1} = [A_1, L_\lambda ]+ \frac{\partial A_1}{\partial \lambda }, \qquad
\frac{\partial L_\lambda }{\partial t_2} = [A_2, L_\lambda ]+ \frac{\partial A_2}{\partial \lambda },
\label{PDE}
\end{gather}
which are equations of $(V_2,V_3,W_1,\widetilde{W}_2)$ with two independent variables $t_1$, $t_2$, where
\begin{gather*}
L_\lambda = \lambda ^3 \left(
\begin{matrix}
1 & 0 \\
0 & -1
\end{matrix}
\right) + \lambda^2 \left(
\begin{matrix}
0 & V_2 \\
1 & 0
\end{matrix}
\right) + \lambda \left(
\begin{matrix}
2t_2-\frac{1}{2}V_2 & V_3 \\
W_1 & -\big(2t_2-\frac{1}{2}V_2\big)
\end{matrix}
\right) \\
\hphantom{L_\lambda =}{} + \left(
\begin{matrix}
t_1-\frac{1}{2}V_3-\frac{1}{2}V_2W_1 & \alpha _4 + t_2V_2 - \frac{1}{4}V_2^2 - V_2\widetilde{W}_2 - V_3W_1 \\
\widetilde{W}_2+t_2 & -\big(t_1-\frac{1}{2}V_3-\frac{1}{2}V_2W_1\big)
\end{matrix}
\right), \\
A_1 = \lambda \left(
\begin{matrix}
1 & 0 \\
0 & -1
\end{matrix}
\right) + \left(
\begin{matrix}
-W_1 & V_2 \\
1 & W_1
\end{matrix}
\right), \\
A_2 = \lambda^2 \left(
\begin{matrix}
1 & 0 \\
0 & -1
\end{matrix}
\right) + \lambda \left(
\begin{matrix}
0 & V_2 \\
1 & 0
\end{matrix}
\right) +\left(
\begin{matrix}
t_2-\frac{1}{2}V_2-\widetilde{W}_2 & V_3 \\
W_1 & -\big(t_2-\frac{1}{2}V_2-\widetilde{W}_2\big)
\end{matrix}
\right).
\end{gather*}
The Poisson tensor $\widetilde{P}_4$ on the symplectic leaf written in the coordinates $(V_2,V_3,W_1,\widetilde{W}_2)$ is given by
\begin{gather*}
\widetilde{P}_4 = \left(
\begin{matrix}
0 &0 &0 &2 \\
0&0 &2 &0 \\
0&-2 &0 &0 \\
-2&0 &0 &0 \\
\end{matrix}
\right).
\end{gather*}
The above two isomonodromic deformation equations can be written as Hamiltonian systems. The Hamiltonian functions of these equations are obtained by deleting $U_1$, $U_2$, $U_3$, $V_4$ from~$\widetilde{\psi}_5$ and~$\widetilde{\psi}_6$ by using the above relations, and changing to Darboux's coordinates by a scaling so that the above~$\widetilde{P}_4$ is transformed to the canonical symplectic matrix. In this manner, we obtain Hamiltonian functions~(\ref{4dimP22}) of~$(\text{P}_\text{II-2})_2$ given in Section~\ref{section1}.

(ii) Consider the pair of vector f\/ields $\widetilde{P}_3 d\widetilde{\psi}_4$ and $\widetilde{P}_3 d\widetilde{\psi}_5$. For the Poisson tensor $\widetilde{P}_3$, a symplectic leaf is def\/ined by the level surface $\{ \widetilde{\psi}_j = \text{const},\, j=1,2,3,6\}$. For the two conditions
\begin{gather}
\frac{\partial \widetilde{X}_\lambda }{\partial \alpha } = \lambda \frac{\partial A_1 }{\partial \lambda } = \lambda \left(
\begin{matrix}
1 & 0 \\
0 & -1
\end{matrix}
\right), \qquad
\frac{\partial \widetilde{X}_\lambda }{\partial \alpha' } = \lambda \frac{\partial A_2 }{\partial \lambda } = 2\lambda^2 \left(
\begin{matrix}
1 & 0 \\
0 & -1
\end{matrix}
\right) + \lambda \left(
\begin{matrix}
U_1 & V_2 \\
1 & -U_1
\end{matrix}
\right), \label{4-2}
\end{gather}
we f\/ind that~$U_2$ in $\widetilde{X}_\lambda $ has to include a parameter~$\alpha $, and~$U_1$ and~$W_1$ have to include a parame\-ter~$\alpha' $. For this purpose, we take the symplectic leaf
\begin{gather*}
S = \big\{ \widetilde{\psi}_1 = 4\alpha _1,\, \widetilde{\psi}_2 = 2\alpha _2+6\alpha _1^2 ,\, \widetilde{\psi}_3 = \alpha _3+4\alpha _1(\alpha _2+\alpha _1^2),\, \widetilde{\psi}_6 = \alpha _6 \big\}.
\end{gather*}
Further, we change the coordinate as $W_1 = \widetilde{W}_1 + \alpha _1$. Then, the above relations for $S$ yield
\begin{gather*}
 U_1 = 2\alpha _1,\qquad U_2 = \alpha _2+\alpha _1^2-\frac{1}{2}V_2,\qquad V_3 = \alpha _3- 2U_3-V_2\widetilde{W}_1 + \alpha _1V_2, \\
 V_4 = \big(\alpha _6-U_3^2\big)/W_2.
\end{gather*}
Substituting them into $\widetilde{X}_\lambda$, $A_1$ and $A_2$, we can verify the condition~(\ref{4-2}) with $\alpha = \alpha _2$ and $\alpha ' = \alpha _1$. By replacing $\alpha _2$, $\alpha _1$ by~$t_1$,~$t_2$, respectively, we obtain the isomonodromic deformation equations~(\ref{PDE}), which are equations of $(V_2,\widetilde{W}_1 ,U_3, W_2)$ with two independent variables~$t_1$, $t_2$. The Poisson tensor~$\widetilde{P}_3$ on the symplectic leaf expressed in the coordinates $(V_2,\widetilde{W}_1 ,U_3, W_2)$ is given by
\begin{gather*}
\widetilde{P}_3 = \left(
\begin{matrix}
0 &2 &0 &0 \\
-2&0 &0 &0 \\
0&0 &0 &W_2 \\
0&0 &-W_2 &0 \\
\end{matrix}
\right).
\end{gather*}
To change to Darboux's coordinates, put
\begin{gather*}
\big(V_2,\widetilde{W}_1 ,U_3, W_2\big) = (2p_1, q_1, q_2p_2- \beta_3, p_2),
\end{gather*}
where $\beta_3$ is an arbitrary parameter. In the new coordinates, we obtain
\begin{gather}
\widetilde{P}_3 = \left(
\begin{matrix}
0 &1 &0 &0 \\
-1&0 &0 &0 \\
0&0 &0 &1 \\
0&0 &-1 &0 \\
\end{matrix}
\right).\label{P}
\end{gather}
Therefore, the two isomonodromic deformation equations (\ref{PDE}) are Hamiltonian systems in this coordinate system. The Hamiltonian functions are obtained by deleting $(U_1,U_2,V_3,V_4)$ from~$\widetilde{\psi}_4$ and~$\widetilde{\psi}_5$ and by changing to the coordinates $(p_1, q_1, p_2, q_2)$. It is easy to verify that if we set $\beta_3^2=\alpha _6$, then two functions become polynomials, which give Hamiltonian functions~(\ref{4dimP4}) of~$(\text{P}_\text{IV})_2$ given in Section~\ref{section1}.

(iii) Consider the pair of vector f\/ields $\widetilde{P}_2 d\widetilde{\psi}_3$ and $\widetilde{P}_2 d\widetilde{\psi}_4$. For the Poisson tensor $\widetilde{P}_2$, a~symplectic leaf is def\/ined by the level surface $\{ \widetilde{\psi}_j = \text{const},\, j=1,2,5,6\}$. In this case, we cannot f\/ind an integer $l$ and a parameter $\alpha '$ such that the condition
\begin{gather*}
\frac{\partial \widetilde{X}_\lambda }{\partial \alpha'} = \lambda^l \frac{\partial A_2 }{\partial \lambda }
\end{gather*}
holds. Hence, we impose only one condition
\begin{gather}
\frac{\partial \widetilde{X}_\lambda }{\partial \alpha} = \lambda^2 \frac{\partial A_1 }{\partial \lambda } = \lambda^2 \left(
\begin{matrix}
1 & 0 \\
0 & -1
\end{matrix}
\right).
\label{4-3}
\end{gather}
For it, $U_1$ in $\widetilde{X}_\lambda $ has to include a parameter $\alpha $. To this end, take the symplectic leaf
\begin{gather*}
S = \big\{ \widetilde{\psi}_1 = 2\alpha _1,\, \widetilde{\psi}_2 = 2\alpha _2+\alpha _1^2 ,\, \widetilde{\psi}_5 = \alpha _5,\, \widetilde{\psi}_6 = \alpha _6 \big\}.
\end{gather*}
This is rearranged as
\begin{gather*}
U_1 = \alpha _1,\qquad V_2 = 2\alpha _2 - 2U_2,\qquad V_3 = (\alpha _5 - 2U_2U_3 - V_4W_1)/W_2, \\
 V_4 = \big(\alpha _6 - U_3^2\big)/W_2.
\end{gather*}
Substituting them into $\widetilde{X}_\lambda$ and $A_1$, we can verify the condition (\ref{4-3}) with $\alpha = \alpha _1$. By repla\-cing~$\alpha _1$ by~$t$, we obtain the isomonodromic deformation equation
\begin{gather*}
\frac{\partial L_\lambda }{\partial t} = [A_1, L_\lambda ]+ \frac{\partial A_1}{\partial \lambda } .
\end{gather*}
The Poisson tensor $\widetilde{P}_2$ on the symplectic leaf with coordinates $(U_2, W_1 ,U_3, W_2)$ is given by
\begin{gather*}
\widetilde{P}_2 = \left(
\begin{matrix}
0 &W_1 &0 & W_2 \\
-W_1&0 &-W_2 &0 \\
0&W_2 &0 &0 \\
-W_2&0 &0 &0 \\
\end{matrix}
\right).
\end{gather*}
To change to Darboux's coordinates, put
\begin{gather*}
(U_2, W_1, U_3, W_2) = (p_1q_1 + p_2q_2 - \beta_2, q_1, p_1q_2-\beta_3, q_2),
\end{gather*}
where $\beta_2$ and $\beta_3$ are arbitrary parameters. In the new coordinates $(q_1, p_1,q_2,p_2)$, $\widetilde{P}_2$ is reduced to the same form as~(\ref{P}). The isomonodromic deformation equation is a Hamiltonian system whose Hamiltonian function is $\widetilde{\psi}_3$ written in this coordinate system. It is easy to verify that if we set $\alpha _6 = \beta_3^2$ and $\alpha _5 = 2\beta_2\beta_3$, then $\widetilde{\psi}_3$ written in the coordinates $(q_1, p_1,q_2,p_2)$ becomes a~polynomial. In this manner, the Hamiltonian function~(\ref{1120}) given in Section~\ref{section1} is obtained.

\subsection{Case (II)}\label{section4.2}

In this case, the functions $\psi _{i,j} = \psi _j$ are given by
\begin{gather*}
\psi _1 = v_1, \\
\psi _2 = u_1^2 + v_2 + v_1w_1, \\
\psi _3 = 2u_1u_2 + v_3 + v_2w_1 + v_1w_2, \\
\psi _4 = u_2^2 + 2u_1u_3 + v_3w_1 + v_2w_2 + v_1w_3, \\
\psi_5 = 2u_2u_3 + v_3w_2 + v_2w_3, \\
\psi_6 = u_3^2 + v_3w_3.
\end{gather*}
The dif\/ferential equation (\ref{D}) def\/ining the distribution~$D$ is $u_j' = -v_j$, $v_j' = 0$, $w_j' = 2u_j$ for $j=1,2,3$. We can assume without loss of generality that $v_1 = 1$ by a suitable scaling of variables. These equations are solved with respect to~$u_1$ as
\begin{gather*}
v_1 = 1,\qquad w_1 = -u_1^2 + W_2,\\
u_2 = V_2 u_1 + U_3,\qquad v_2 = V_2, \qquad w_2 = -V_2u_1^2-2U_3u_1+W_4, \\
u_3 = V_4 u_1 + U_5,\qquad v_3 = V_4,\qquad w_3 = -V_4u_1^2 - 2U_5u_1 + W_6,
\end{gather*}
where $W_2$, $U_3$, $V_2$, $W_4$, $U_5$, $V_4$, $W_6$ are integral constants (initial values at $u_1 = 0$). This relation def\/ines a coordinate transformation
\begin{gather*}
(u_1, w_1, u_2, v_2, w_2, u_3, v_3, w_3) \mapsto (u_1, W_2, U_3, V_2, W_4, U_5, V_4, W_6).
\end{gather*}
In the new coordinates, integral manifolds of the distribution $D$ are straight lines along $u_1$-axis. In particular, the subset $\{ u_1 = 0\} \subset \mathfrak{g}^0_n$ gives the realization of the orbit space $\mathfrak{g}^0_n/D$ as a submanifold and $( W_2, U_3, V_2, W_4, U_5, V_4, W_6)$ provides a global coordinate system of $\mathfrak{g}^0_n/D$ restricted to $v_1 = 1$.

On $\mathfrak{g}^0_n/D$, we have functions
\begin{gather*}
\widetilde{\psi}_1 = 1, \\
\widetilde{\psi}_2 = V_2+W_2, \\
\widetilde{\psi}_3 = V_4+W_4+V_2W_2, \\
\widetilde{\psi}_4 = U_3^2 + W_6 + V_2W_4 + V_4W_2, \\
\widetilde{\psi}_5 = 2U_3U_5 + V_2W_6 + V_4W_4, \\
\widetilde{\psi}_6 = U_5^2 + V_4W_6,
\end{gather*}
and two vector f\/ields (\ref{vec1}), (\ref{vec2}) expressed in Lax form (\ref{Lax4}) with
\begin{gather*}
\widetilde{X}_\lambda = \lambda ^3 \left(
\begin{matrix}
0 & 0 \\
1 & 0
\end{matrix}
\right) + \lambda^2 \left(
\begin{matrix}
0 & 1 \\
W_2 & 0
\end{matrix}
\right) + \lambda \left(
\begin{matrix}
U_3 & V_2 \\
W_4 & -U_3
\end{matrix}
\right) + \left(
\begin{matrix}
U_5 & V_4 \\
W_6 & -U_5
\end{matrix}
\right), \\
A_1 = \lambda \left(
\begin{matrix}
0 & 0 \\
1 & 0
\end{matrix}
\right) + \left(
\begin{matrix}
0 & 1 \\
W_2 & 0
\end{matrix}
\right) - \left(
\begin{matrix}
0 & 0 \\
V_2 & 0
\end{matrix}
\right), \\
A_2 = \lambda^2 \left(
\begin{matrix}
0 & 0 \\
1 & 0
\end{matrix}
\right) + \lambda \left(
\begin{matrix}
0 & 1 \\
W_2 & 0
\end{matrix}
\right) + \left(
\begin{matrix}
U_3 & V_2 \\
W_4 & -U_3
\end{matrix}
\right)
 - \left(
\begin{matrix}
0 & 0 \\
V_4 & 0
\end{matrix}
\right).
\end{gather*}

The next purpose is to restrict the vector f\/ields on a symplectic leaf.

(i) Consider the pair of vector f\/ields $\widetilde{P}_4 d\widetilde{\psi}_5$ and $\widetilde{P}_4 d\widetilde{\psi}_6$. For the Poisson tensor $\widetilde{P}_4$, a symplectic leaf is def\/ined by the level surface $\{ \widetilde{\psi}_j = \text{const},\, j=1,2,3,4\}$. In order for the two conditions
\begin{gather}
\frac{\partial \widetilde{X}_\lambda }{\partial \alpha } = \frac{\partial A_1 }{\partial \lambda } = \left(
\begin{matrix}
0 & 0 \\
1 & 0
\end{matrix}
\right), \qquad
\frac{\partial \widetilde{X}_\lambda }{\partial \alpha' } = \frac{\partial A_2 }{\partial \lambda } = 2\lambda \left(
\begin{matrix}
0 & 0 \\
1 & 0
\end{matrix}
\right) + \left(
\begin{matrix}
0 & 1 \\
W_2 & 0
\end{matrix}
\right)\label{4-7}
\end{gather}
to be satisf\/ied, we f\/ind that $W_6$ in $\widetilde{X}_\lambda $ has to include a parameter~$\alpha $, which will be replaced by~$t_1$ later, and $W_4$ and $V_4$ have to include a parameter $\alpha' $, which will be replaced by~$t_2$ later. For this purpose, we take the symplectic leaf
\begin{gather*}
S = \big\{ \widetilde{\psi}_1 = 1,\, \widetilde{\psi}_2 = 0 ,\, \widetilde{\psi}_3 = 3\alpha _4,\,\widetilde{\psi}_4 = \alpha _6\big\}.
\end{gather*}
Further, we change the coordinate as $W_4 = \widetilde{W}_4 + 2\alpha _4$. Then, the above relation for $S$ is rearranged as
\begin{gather*}
 V_2 = -W_2,\qquad V_4 = \alpha _4-\widetilde{W}_4+W_2^2,\qquad W_6 = \alpha _6 - U_3^2 + 2W_2\widetilde{W}_4 - W_2^3 + \alpha _4W_2.
\end{gather*}
Substituting them into $\widetilde{X}_\lambda$, $A_1$ and $A_2$, we can verify the condition~(\ref{4-7}) with $\alpha = \alpha _6$ and $\alpha ' = \alpha _4$. By replacing $\alpha _6$,~$\alpha _4$ by~$t_1$,~$t_2$, respectively, we obtain the isomonodromic deformation equations~(\ref{PDE}), which are equations of $(W_2,U_3,\widetilde{W}_4, U_5)$ with two independent variables~$t_1$,~$t_2$. The Poisson tensor $\widetilde{P}_4$ on the symplectic leaf written in this coordinate system is given by
\begin{gather*}
\widetilde{P}_4 = \left(
\begin{matrix}
0 &0 &0 &1 \\
0&0 & -1 & 0 \\
0&1 &0 & W_2 \\
-1&0 &-W_2 &0 \\
\end{matrix}
\right).
\end{gather*}
To change to Darboux's coordinates, put
\begin{gather*}
\big(W_2,U_3,\widetilde{W}_4, U_5\big) = (q_1, p_2, q_2, p_1 + p_2q_1).
\end{gather*}
Then, $\widetilde{P}_4$ is transformed to the canonical symplectic matrix. In the coordinates $(q_1, p_1, q_2, p_2)$, the above two isomonodromic deformation equations are Hamiltonian systems. The Hamiltonian functions are obtained by deleting $V_2$, $V_4$, $W_6$ from $\widetilde{\psi}_5$ and $\widetilde{\psi}_6$ by using the above relations, and changing to Darboux's coordinates. In this manner, we obtain Hamiltonian functions~(\ref{4dimP1}) of~$(\text{P}_\text{I})_2$ given in Section~\ref{section1}.

(ii) Consider the pair of vector f\/ields $\widetilde{P}_3 d\widetilde{\psi}_4$ and $\widetilde{P}_3 d\widetilde{\psi}_5$. For the Poisson tensor $\widetilde{P}_3$, a symplectic leaf is def\/ined by the level surface $\{ \widetilde{\psi}_j = \text{const},\, j=1,2,3,6\}$.
For the two conditions
\begin{gather}
\frac{\partial \widetilde{X}_\lambda }{\partial \alpha } = \lambda \frac{\partial A_1 }{\partial \lambda } = \lambda \left(
\begin{matrix}
0 & 0 \\
1 & 0
\end{matrix}
\right), \qquad
\frac{\partial \widetilde{X}_\lambda }{\partial \alpha' } = \lambda \frac{\partial A_2 }{\partial \lambda } = 2\lambda^2 \left(
\begin{matrix}
0 & 0 \\
1 & 0
\end{matrix}
\right) + \lambda \left(
\begin{matrix}
0 & 1 \\
W_2 & 0
\end{matrix}
\right), \label{4-8}
\end{gather}
we f\/ind that $W_4$ in $\widetilde{X}_\lambda $ has to include a parameter~$\alpha $, and $W_2$ and $V_2$ have to include a parame\-ter~$\alpha' $. For this purpose, we take the symplectic leaf
\begin{gather*}
S = \big\{ \widetilde{\psi}_1 = 1,\, \widetilde{\psi}_2 = 3\alpha _2
 ,\, \widetilde{\psi}_3 = \alpha _4 + 3\alpha _2^2,\,
 \widetilde{\psi}_6 = \alpha _{10}\big \}.
\end{gather*}
Further, \looseness=-1 we change the coordinate as $V_2 = \widetilde{V}_2 + \alpha _2$. Then, the above relation for $S$ is rearranged as
\begin{gather*}
W_2 = 2 \alpha _2-\widetilde{V}_2,\qquad W_4 = \alpha _4 + \alpha _2^2 - \alpha _2\widetilde{V}_2 - V_4 + V_2^2,
\qquad W_6 = \big(\alpha _{10}- U_5^2\big)/V_4.
\end{gather*}
Substituting them into $\widetilde{X}_\lambda , A_1$ and $A_2$, we can verify the condition~(\ref{4-8}) with $\alpha = \alpha _4$ and $\alpha ' = \alpha _2$. By replacing~$\alpha _4$,~$\alpha _2$ by~$t_1$,~$t_2$, respectively, we obtain the isomonodromic deformation equations~(\ref{PDE}). The Poisson tensor $\widetilde{P}_3$ on the symplectic leaf written in the coordinates $(\widetilde{V}_2, U_3, V_4, U_5)$ is given by
\begin{gather*}
\widetilde{P}_3 = \left(
\begin{matrix}
0 &1 &0 &0 \\
-1&0 &0 &0 \\
0&0 &0 &-V_4 \\
0&0 &V_4 &0 \\
\end{matrix}
\right).
\end{gather*}
To change to Darboux's coordinates, put
\begin{gather*}
\big(\widetilde{V}_2, U_3, V_4, U_5\big) = (p_2, q_2, p_1, q_1p_1 - \alpha _5),
\end{gather*}
where $\alpha _5$ is an arbitrary parameter. Then, $\widetilde{P}_3$ is transformed to the canonical symplectic mat\-rix. In the coordinates $(q_1, p_1, q_2, p_2)$, the above two isomonodromic deformation equations are Hamiltonian systems. The Hamiltonian functions are obtained by deleting $W_2$, $W_4$, $W_6$ from~$\widetilde{\psi}_4$ and~$\widetilde{\psi}_5$ by using the above relations, and by changing to Darboux's coordinates. It is easy to verify that if we set $\alpha _{10}=\alpha _5^2$, then two functions become polynomials. In this manner, we obtain Hamiltonian functions (\ref{4dimP21}) of $(\text{P}_\text{II-1})_2$ given in Section~\ref{section1}.

(iii) Consider the vector f\/ields $\widetilde{P}_2 d\widetilde{\psi}_3$ and $\widetilde{P}_2 d\widetilde{\psi}_4$. For the Poisson tensor $\widetilde{P}_2$, a symplectic leaf is def\/ined by the level surface $\{ \widetilde{\psi}_j = \text{const},\, j=1,2,5,6\}$. In this case, we cannot f\/ind an integer $l$ and a parameter $\alpha '$ such that the condition
\begin{gather*}
\frac{\partial \widetilde{X}_\lambda }{\partial \alpha'} = \lambda^l \frac{\partial A_2 }{\partial \lambda }
\end{gather*}
holds. Hence, we impose only one condition
\begin{gather}
\frac{\partial \widetilde{X}_\lambda }{\partial \alpha} = \lambda^2 \frac{\partial A_1 }{\partial \lambda } = \lambda^2 \left(
\begin{matrix}
0 & 0 \\
1 & 0
\end{matrix}
\right).\label{4-9}
\end{gather}
For it, $W_2$ in $\widetilde{X}_\lambda $ has to include a parameter~$\alpha $. To this end, take the symplectic leaf
\begin{gather*}
S = \big\{ \widetilde{\psi}_1 = 1,\, \widetilde{\psi}_2 = \alpha _2 ,\, \widetilde{\psi}_5 = \alpha _{8},\, \widetilde{\psi}_6 = \alpha _{10} \big\}.
\end{gather*}
This is rearranged as
\begin{gather*}
W_2 = \alpha _2 - V_2,\qquad W_4 = (\alpha _8 - 2U_3U_5 - V_2W_6)/V_4,\qquad W_6 = \big(\alpha _{10} - U_5^2\big)/V_4.
\end{gather*}
Substituting them into $\widetilde{X}_\lambda$ and $A_1$, we can verify the condition~(\ref{4-9}) with $\alpha = \alpha _2$. By re\-pla\-cing~$\alpha _2$ by~$t$, we obtain the isomonodromic deformation equation
\begin{gather*}
\frac{\partial L_\lambda }{\partial t} = [A_1, L_\lambda ]+ \frac{\partial A_1}{\partial \lambda } .
\end{gather*}
The Poisson tensor $\widetilde{P}_2$ on the symplectic leaf with coordinates $(U_3, V_2 ,U_5, V_4)$ is given by
\begin{gather*}
\widetilde{P}_2 = \left(
\begin{matrix}
0 &-V_2 &0 & -V_4 \\
V_2 &0 & V_4 &0 \\
0& -V_4 &0 &0 \\
V_4&0 &0 &0 \\
\end{matrix}
\right).
\end{gather*}
To change to Darboux's coordinates, put
\begin{gather*}
(U_3, V_2 ,U_5, V_4) = (p_1q_1 + p_2q_2 - \beta_3,\, p_2,\, p_1q_2-\beta_5,\, p_1),
\end{gather*}
where $\beta_3$ and $\beta_5$ are arbitrary parameters. Then, $\widetilde{P}_2$ is transformed to the canonical symplectic matrix. In the coordinates $(q_1, p_1, q_2, p_2)$, the above isomonodromic deformation equation is a~Hamiltonian system. The Hamiltonian function is obtained by deleting $W_2$, $W_4$, $W_6$ from $\widetilde{\psi}_3$ by using the above relations, and by changing to Darboux's coordinates. It is easy to verify that if we set $\alpha _{10}=\beta_5^2$ and $\alpha _8 = 2\beta_3 \beta_5$, then $\widetilde{\psi}_3$ written in the coordinates $(q_1, p_1,q_2,p_2)$ becomes a~polynomial. This procedure yields the Hamiltonian function (\ref{-1412}) given in Section~\ref{section1}.

\section[4-dimensional Painlev\'{e} equations: $\mathfrak{g} = \mathfrak{so}_5$, $n=1$]{4-dimensional Painlev\'{e} equations: $\boldsymbol{\mathfrak{g} = \mathfrak{so}_5}$, $\boldsymbol{n=1}$}\label{section5}

According to \cite{DS}, we use the following representation for the Lie algebra $\mathfrak{g} \simeq \mathfrak{so}_5$ of type B$_2$
\begin{gather*}
X_i = \left(
\begin{matrix}
p_i & q_i & r_i & s_i & 0 \\
t_i & u_i & v_i & 0 & s_i \\
w_i & x_i & 0 & v_i & -r_i \\
y_i & 0 & x_i & -u_i & q_i \\
0 & y_i & -w_i & t_i & -p_i
\end{matrix}
\right).
\end{gather*}
Consider the Lie algebra $\mathfrak{g}_1 = \{ X_\lambda = \lambda X_0 + X_1 \, | \, X_i\in \mathfrak{g} \simeq \mathfrak{so}_5\}$. For the def\/inition of $\mathfrak{g}^0_1$, we only consider the following case
\begin{gather*}
X_0 = \left(
\begin{matrix}
0 & 0 & 0 & 0 & 0 \\
1 & 0 & 0 & 0 & 0 \\
0 & 1 & 0 & 0 & 0 \\
0 & 0 & 1 & 0 & 0 \\
0 & 0 & 0 & 1 & 0 \\
\end{matrix}
\right),
\end{gather*} (i.e., $x_0=t_0=1$ and zeros otherwise). The purpose in this section is to derive the Hamiltonian~(\ref{4dimCos}) for Cosgrove's equation. The other choice of $X_0$ may yield dif\/ferent Painlev\'{e} systems. Note that $n=1$, $\operatorname{rank}(\mathfrak{g}) = h=2$, $m_1 = 1$ and $m_2 = 3$.
We have the following functions
\begin{gather*}
\psi_{1,1} = -2q_1 - 2v_1, \\
\psi_{1,2} = -2q_1t_1 - u_1^2 - 2v_1x_1 - 2s_1y_1, \\
\psi_{2,1} = -2s_1, \\
\psi_{2,2} = q_1^2 - 2s_1t_1 + 2q_1v_1 - 4s_1x_1, \\
\psi_{2,3} = 2q_1^2t_1 + 2q_1t_1v_1 + 2s_1u_1w_1 - 4s_1t_1x_1 + 2q_1v_1x_1
 -2s_1x_1^2 - 2q_1s_1y_1 + 2s_1v_1y_1, \\
\psi_{2,4} = q_1^2t_1^2 - 2q_1u_1v_1w_1 + 2q_1s_1w_1^2 + 2q_1t_1v_1x_1 + 2s_1u_1w_1x_1 \\
\hphantom{\psi_{2,4} =}{} -2s_1t_1x_1^2 - 2q_1s_1t_1y_1 - 2q_1v_1^2y_1 + 2s_1v_1x_1y_1 + s_1^2y_1^2,
\end{gather*}
which are coef\/f\/icients of the characteristic polynomial $\det (\mu - X_\lambda )$.

We solve the dif\/ferential equations for the two dimensional distribution~$D$ as functions of $(p_1, r_1)$ with the initial condition $(q_1, s_1, t_1, u_1, v_1, w_1, x_1, y_1) =(Q,S,T,U,V,W,X,Y)$ at $(p_1, r_1) = (0,0)$. The expressions of solutions are too long and omitted here. These solutions def\/ine a coordinate transformation
\begin{gather*}
(p_1, r_1, q_1, s_1, t_1, u_1, v_1, w_1, x_1, y_1) \mapsto (p_1, r_1, Q,S,T,U,V,W,X,Y).
\end{gather*}
In the new coordinates, integral manifolds of the distribution~$D$ are plains which are parallel to the $(p_1, r_1)$-plain. In particular, the subset $\{ p_1 = r_1 = 0\} \subset \mathfrak{g}^0_1$ gives the realization of the orbit space~$\mathfrak{g}^0_1/D$ as a~submanifold and $(Q,S,T,U,V,W,X,Y)$ provides a~global coordinate system of~$\mathfrak{g}^0_1/D$.

At this stage, we have on $\mathfrak{g}^0_1/D$
\begin{gather*}
\widetilde{\psi}_{1,1} = -2Q - 2V, \\
\widetilde{\psi}_{1,2} = -2QT - U^2 - 2VX - 2SY, \\
\widetilde{\psi}_{2,1} = -2S, \\
\widetilde{\psi}_{2,2} = Q^2 - 2ST + 2QV- 4SX, \\
\widetilde{\psi}_{2,3} = 2Q^2T + 2QTV + 2SUW - 4STX+ 2QVX -2SX^2 - 2QSY + 2SVY,
\end{gather*}
and the vector f\/ield $\widetilde{P}_2d\widetilde{\psi}_{2,3}$, whose Casimir functions are $\widetilde{\psi}_{1,1}$, $\widetilde{\psi}_{1,2}$, $\widetilde{\psi}_{2,1}$ and $\widetilde{\psi}_{2,2}$. The corresponding dif\/ferential equation is expressed in Lax form as $d\widetilde{X}_\lambda /dt = [A_\lambda , \widetilde{X}_\lambda ]$, where
\begin{gather*}
X_\lambda = \lambda \left(
\begin{matrix}
0 & 0 & 0 & 0 & 0 \\
1 & 0 & 0 & 0 & 0 \\
0 & 1 & 0 & 0 & 0 \\
0 & 0 & 1 & 0 & 0 \\
0 & 0 & 0 & 1 & 0 \\
\end{matrix}
\right) + \left(
\begin{matrix}
0 & Q & 0 & S & 0 \\
T & U & V & 0 & S \\
W & X & 0 & V & 0 \\
Y & 0 & X & -U & Q \\
0 & Y & -W & T & 0 \\
\end{matrix}
\right), \\
 A_\lambda = \lambda \nabla _1\widetilde{\psi}_{2,1}+\nabla _1\widetilde{\psi}_{2,2}
 +(V-Q) \nabla _1\widetilde{\psi}_{1,1} + \frac{2}{S}\big(Q^2-SX\big) \nabla _1\widetilde{\psi}_{2,1}, \\
\nabla _1\widetilde{\psi}_{2,1} = \left(
\begin{matrix}
0 & 0 & 0 & 0 & 0 \\
0 & 0 & 0 & 0 & 0 \\
0 & 0 & 0 & 0 & 0 \\
-2 & 0 & 0 & 0 & 0 \\
0 & -2 & 0 & 0 & 0 \\
\end{matrix}
\right),\qquad \nabla _1\widetilde{\psi}_{1,1}
= \left(
\begin{matrix}
0 & 0 & 0 & 0 & 0 \\
-2 & 0 & 0 & 0 & 0 \\
0 & -2 & 0 & 0 & 0 \\
0 & 0 & -2 & 0 & 0 \\
0 & 0 & 0 & -2 & 0 \\
\end{matrix}
\right), \\
 \nabla _1\widetilde{\psi}_{2,2}=\left(
\begin{matrix}
0&-2S &0 &0 &0 \\
2(Q+V)& 0 &-4S &0 &0 \\
-2U&2Q &0 &-4S &0 \\
-2(T+2X)&0 &2Q &0 &-2S \\
0&-2(T+2X) &2U &2(Q+V) &0 \\
\end{matrix}
\right).
\end{gather*}

The next purpose is to restrict the vector f\/ield on a symplectic leaf. For the Poisson tensor~$\widetilde{P}_2$, a symplectic leaf is def\/ined by the level surface $\{ \widetilde{\psi}_{i,j} = \text{const},\, i,j=1,2\}$. In order for the condition
\begin{gather}
\frac{\partial \widetilde{X}_\lambda }{\partial \alpha } = \frac{\partial A_\lambda }{\partial \lambda } =\left(
\begin{matrix}
0 & 0 & 0 & 0 & 0 \\
0 & 0 & 0 & 0 & 0 \\
0 & 0 & 0 & 0 & 0 \\
-2 & 0 & 0 & 0 & 0 \\
0 & -2 & 0 & 0 & 0 \\
\end{matrix}\right)\label{5-1}
\end{gather}
to be satisf\/ied, we f\/ind that $Y$ in $\widetilde{X}_\lambda $ has to include a parameter $\alpha $ which will be replaced by~$t$ later. For this purpose, we take the symplectic leaf
\begin{gather*}
S = \big\{ \widetilde{\psi}_{2,1} = -2,\, \widetilde{\psi}_{1,1} = -2\alpha _2,\,\widetilde{\psi}_{2,2} = -2\alpha _4 + \alpha _2^2,\, \widetilde{\psi}_{1,2} = 4\alpha _6 - 2\alpha _2\alpha _4\big \}.
\end{gather*}
This is rewritten as
\begin{gather*}
\begin{split}
& S = 1,\qquad Q = \alpha _2-V,\qquad T = \alpha _4-2X-\frac{1}{2}V^2, \\
& Y = -2\alpha _6-\frac{1}{2}U^2 - \frac{1}{2}V^3 - 3VX + \frac{\alpha _2}{2}V^2 + 2\alpha _2X + \alpha _4V.
\end{split}
\end{gather*}
Substituting them into $\widetilde{X}_\lambda $ and $A_\lambda $, it turns out that the condition~(\ref{5-1}) is satisf\/ied with $\alpha =\alpha _6$. Finally, by replacing $\alpha _6$ by $t$, we obtain the isomonodromic deformation equation~(\ref{2-9}). The Poisson tensor $\widetilde{P}_2$ written with respect to the coordinates $(U,V,W,X)$ is already in the canonical symplectic matrix. Thus, the isomonodromic deformation equation is a Hamiltonian system with the Hamiltonian function $\widetilde{\psi}_{2,3}$ written in the coordinate system $(U,V,W,X)$. Since this expression is too complicated, we further introduce the symplectic transformation
\begin{gather*}
 (U,V,W,X) = \left(p_2,\, q_1+\frac{13}{18}\alpha _2,\, p_1 + \frac{4}{13}p_2\left(q_1+\frac{13}{18}\alpha _2\right)
,\right.\\
\left.\hphantom{(U,V,W,X) =}{} q_2+\frac{1}{3}\alpha _4 + \frac{7}{108}\alpha _2^2-\frac{2}{13}\left(q_1+\frac{13}{18}\alpha _2\right)^2\right).
\end{gather*}
Then, the Hamiltonian function (\ref{4dimCos2}), which is equivalent to (\ref{4dimCos}), is obtained.

\pdfbookmark[1]{References}{ref}
\LastPageEnding

\end{document}